\documentclass[final]{siamltex}

\usepackage{amsmath,amssymb,amsxtra,comment,graphicx,psfrag}
\usepackage{bm,mathrsfs}
\usepackage{booktabs,subcaption,amsfonts,dcolumn}
\usepackage{mathtools}
\usepackage{color,xcolor}
\usepackage{enumerate,enumitem}
\usepackage{pdfsync}
\usepackage{hhline}
\usepackage{cite}
\usepackage{bm}
\usepackage{graphicx}

\newtheorem{remark}{Remark}[section]
\newtheorem{assumption}{Assumption}[section]
\newtheorem{example}{Example}[section]
\newcommand{\fy}{\varphi}
\renewcommand{\d}{{\rm d}}

\def\Om{\Omega}
\def\II{(\Om)}

\def\Uad{\mathcal{A}}

\title{Error Analysis of Finite Element Approximations of Diffusion Coefficient Identification for Elliptic and Parabolic Problems
\thanks{The work of B. Jin is supported by UK EPSRC grant EP/T000864/1, and
the research of Z. Zhou is supported by Hong Kong RGC grant (No. 15304420).}}

\author{
Bangti Jin\thanks{Department of Computer Science, University College London, Gower Street, London, WC1E 6BT, UK.
 E-mail address: b.jin@ucl.ac.uk, bangti.jin@gmail.com}
\and
Zhi Zhou\thanks{Department of Applied Mathematics, The Hong Kong Polytechnic University, Kowloon, Hong Kong. E-mail address: zhizhou@polyu.edu.hk}
}

\begin{document}

\maketitle

\begin{abstract}
In this work, we present a novel error analysis for recovering a spatially dependent diffusion coefficient in
an elliptic or parabolic problem. It is based on the standard regularized output least-squares
formulation with an $H^1(\Omega)$ seminorm penalty, and then discretized using the Galerkin finite element
method with conforming piecewise linear finite elements for both state and coefficient, and backward Euler in
time in the parabolic case. We derive \textit{a priori} weighted $L^2(\Omega)$ estimates where the constants
depend only on the given problem data for both elliptic and parabolic cases. Further, these estimates also
allow deriving standard $L^2(\Omega)$ error estimates, under a positivity condition that can be verified
for certain problem data. Numerical experiments are provided to complement the error analysis.
\end{abstract}



\section{Introduction}\label{sec:intro}
This work is concerned with error analysis of Galerkin approximations of regularized formulations for recovering a
spatially-dependent diffusion coefficient $q$ for elliptic and parabolic problems. Let $\Omega\subset\mathbb{R}^d$
($d=1,2,3$) be a convex polyhedral domain with a boundary $\partial\Omega$. Consider the following elliptic boundary value problem:
\begin{equation}\label{eqn:pde}
 \left\{\begin{aligned}
   -\nabla\cdot(q\nabla u)  &= f,&&  \mbox{in }\Omega ,\\
    u &=  0,&&\mbox{on }\partial\Omega,
   \end{aligned}\right.
\end{equation}
where the function $f$ denotes a given source term. The solution to problem \eqref{eqn:pde} is denoted by $u(q)$, to indicate its
dependence on the coefficient $q$. The inverse problem is to recover the exact diffusion coefficient $q^\dag(x)$
from the pointwise observation $z^\delta$, with a noise level $\delta$, i.e.,
\begin{equation}\label{eqn:noise-ell}
  \|  z^\delta - u(q^\dag)  \|_{L^2(\Omega)} \le \delta.
\end{equation}
{Throughout, the diffusion coefficient $q$ is sought within the admissible set $\Uad$, defined by
\begin{equation}\label{eqn:admset}
    \Uad=\{ q\in H^1(\Omega):~~c_0\le q \le c_1 ~~\text{a.e. in}~~ \Omega\},
\end{equation}
for some positive constants $c_0, c_1>0$.}

Problem \eqref{eqn:pde} is the steady state of the following parabolic initial-boundary value problem
\begin{equation}\label{eqn:parabolic}
    \left\{\begin{aligned} \partial_tu -\nabla\cdot(q\nabla u)  &= f,&&  \mbox{in }\Omega\times(0,T],\\
    u(0) &= u_0, && \mbox{in }\Omega,\\
u &=0,&& \mbox{on }\partial\Omega\times(0,T],
   \end{aligned} \right.
\end{equation}
where $T>0$ is the final time. The functions $f$ and $u_0$ are the given source term and initial condition,
respectively. The corresponding inverse problem is to recover the spatially dependent diffusion coefficient
$q^\dag(x)$ from the distributed observation $z^\delta$ over $\Omega\times(T-\sigma,T)$ (for some measurement
window $0\leq \sigma<T$), with a noise level $\delta$, i.e.,
\begin{equation}\label{eqn:noise-p0}
  \|z^\delta -u(q^\dag)\|_{L^2(T-\sigma,T;L^2(\Omega))} \leq \delta.
\end{equation}
The elliptic problem \eqref{eqn:pde} and parabolic problem \eqref{eqn:parabolic} describe many important
physical processes, and the related inverse problems are exemplary for parameter identifications for
PDEs (see the monographs \cite{BanksKunisch:1989,Isakov:2006} for overviews).
For example, \eqref{eqn:pde} is often used to model the behaviour of a confined inhomogenous aquifer, where $u$
represents the piezometric head, $f$ is the recharge and $q$ hydraulic conductivity (or transmissivity
in the two-dimensional case); see \cite{FrindPinder:1973,Yeh:1986}
for parameter identifications in hydrology. See also \cite{BalUhlmann:2013} for
related coupled-physics inverse problems arising in medical imaging.

Due to the ill-posed nature of inverse problems, regularization, especially Tikhonov regularization, is customarily
employed for constructing numerical approximations (see, e.g., \cite{EnglHankeNeubauer:1996,ItoJin:2015}).
Commonly used stabilizing terms include $H^1(\Omega)$ and total variation {semi-norms}, which are suitable for
recovering smooth and nonsmooth diffusion coefficients, respectively. The well-posedness and convergence (with
respect to the noise level) was studied \cite{KravarisSeinfeld:1985,Gutman:1990,Acar:1993,ChenZou:1999}, and further,
convergence rates (with respect to $\delta$) were derived under various ``source'' conditions, e.g., variational inequalities
or conditional stability estimates \cite{KeungZou:1998}. In practice, the regularized formulations are further discretized, often
with the Galerkin finite element method (FEM), due to its flexibility with domain geometry and low-regularity problem data.
The discretization step necessarily introduces additional errors, which impacts the reconstruction quality. Several
studies  \cite{Gutman:1990,KeungZou:1998,Zou:1998} have analyzed the convergence with respect to the discretization
parameter(s), e.g., mesh size $h$ and time stepsize $\tau$, but without error bounds.

So far, only very few results were available on error bounds of approximate solutions. This is attributed to
strong nonlinearity of the forward (parameter-to-state) map, low regularity of noisy data $z^\delta$
and delicate interplay between different parameters (noise level, regularization parameter and discretization
parameters). Falk \cite{Falk:1983} analyzed a Galerkin discretization of the standard output least-squares
formulation for the elliptic inverse problem (with a Neumann boundary
condition), and derived a rate $O(h^r + h^{-2}\delta)$ in the $L^2(\Omega)$ norm, where $r$ is the polynomial degree of the
finite element space and $h$ is the mesh size. This result is derived by assuming sufficiently high regularity
of the coefficient $q^\dag$, and a certain structural condition on the gradient field;
see details in Remark \ref{rem:cond-ellip}. In the elliptic case, there are also several results for other discrete
formulations: \cite{Richter:1981} for upwind finite difference approximation of a transport equation (without noise),
\cite{Karkkainen:1997,AlJamalGockenbach:2012} for the equation error approach (EEA) (the fidelity in the negative
$H^1(\Omega)$ norm, and $H^1(\Omega)$ penalty) and \cite{KohnLowe:1988} for the EEA in a mixed
formulation. However, no regularization was taken into account in the works \cite{Richter:1981,Falk:1983,KohnLowe:1988},
and thus the corresponding discrete formulations can suffer from numerical instability. The EEA works only
with the case $z^\delta\in H^1(\Omega)$, and so is the error analysis. For the regularized
problem, Wang and Zou \cite{WangZou:2010} derived first convergence rates (in weighted norms) for both elliptic
and parabolic cases ({equipped} with a zero Neumann boundary condition) with either pointwise or
gradient observations. In the elliptic case, the analysis employs
the test function $\fy=\frac{q^\dag-q_h^*}{q^\dag}e^{-2\alpha_0{c_0^{-1}}u(q^\dag)}$ (with $q_h^*$
being a discrete minimizer, $\alpha_0$ a parameter in the structural condition, cf. Remark \ref{rem:cond-ellip}
and $c_0$ lower bound on $q^\dag$),
and assumes regularity on both state $u$ and coefficient $q^\dag$; and in the parabolic case, it requires
a more involved test function. However, no estimate in the usual $L^2(\Omega)$ was given, and further, the analysis
in the parabolic case requires the measurement in the entire time interval $(0,T)$. Deckelnick and Hinze
\cite{DeckelnickHinze:2012} studied the elliptic inverse problem of recovering matrix valued coefficients using
the $L^2(\Omega)$ penalty in the $H$-convergence framework, and in the two-dimensional case, proved an
$L^2(\Omega)$ estimate $O(\delta^\frac12)$, where the coefficient $q$ is discretized using variational discretization.
The estimate was derived under a projected source condition.

In this work, we present a novel approach to derive convergence rates for the standard regularized
output least-squares formulation discretized by Galerkin FEM. The approach employs the test function
$\fy=\frac{q^\dag-q_h^*}{q^\dag}u(q^\dag)$ for both elliptic and parabolic cases, inspired by the
recent work \cite{Bonito:2017} {(on the H\"{o}lder stability of the elliptic inverse problems).}
It enables us to derive convergence rates in a new weighted $L^2(\Omega)$
norm for both elliptic and parabolic cases, {extending the prior result for
the time-fractional diffusion equation \cite{JinZhou:sicon2019}.} Further, we derive estimates in the usual
$L^2(\Omega)$ norm, under suitable positivity conditions, which hold for a class of problems data.
In the parabolic case, we relax the restriction in \cite{WangZou:2010} {(and also \cite{JinZhou:sicon2019})} on the time horizon for
the measurement from $[0,T]$ to a subinterval $[T-\sigma,T]$ for any $0\leq\sigma<T$ and the regularity
assumption on the true coefficient $q^\dag$ from $W^{2,\infty}(\Omega)$ to $W^{1,\infty}(\Omega)\cap H^2(\Omega)$.
This former is achieved by a new weighting in the time direction, and
the latter by discrete maximal $L^p$ regularity for parabolic problems. In the course of error analysis, no regularity assumption
is made on the state $u$ and {no additional temporal} regularity on the observation $z^\delta$ than $L^2(T-\sigma,T;L^2(\Omega))$,
and furthermore, no source type condition is imposed, as usually
done for parameter identifications \cite{EnglKunischNeubauer:1989,KeungZou:1998}. To the best of our knowledge, they are
first error estimates of the kind for the concerned inverse conductivity problems.

The rest of the paper is organized as follows. In Section \ref{sec:fem}, we describe useful facts about
the Galerkin FEM. Then in Sections \ref{sec:elliptic} and \ref{sec:parabolic}, we describe and
analyze the finite element approximations for the elliptic and parabolic inverse problems, respectively.
Last, in Section \ref{sec:numer}, we present numerical
results to complement the analysis. We conclude with useful notation. For any $k\geq 0$ and $p\geq1$, the
space $W^{k,p}(\Omega)$ denotes the standard Sobolev spaces of the $k$th order, and we write $H^k(\Omega)$,
when $p=2$ \cite{AdamsFournier:2003}. The notation $(\cdot,\cdot)$ denotes the $L^2(\Omega)$ inner product.
For the analysis of parabolic problems, we use the Bochner spaces $W^{k,p}(0,T;B)$ etc, with $B$ being a
Banach space. Throughout, the notation $c$, with or without a subscript, denotes a generic constant which may
change at each occurrence, but it is always independent of the following parameters: regularization
parameter $\gamma$, mesh size $h$, time stepsize $\tau$ and noise level $\delta$.

\section{Finite element approximations}\label{sec:fem}
Now we recall briefly the Galerkin FEM approximation.
Let $\mathcal{T}_h$ be a shape regular quasi-uniform triangulation of the
domain $\Omega $ into $d$-simplexes, denoted by $T$, with a mesh size $h$. Over $\mathcal{T}_h$,
we define a continuous piecewise linear finite element space $X_h$ by
\begin{equation*}
  X_h= \left\{v_h\in H_0^1(\Omega):\ v_h|_T \mbox{ is a linear function}\ \forall\, T \in \mathcal{T}_h\right\},
\end{equation*}
and similarly the space $V_h$ by
\begin{equation*}
  V_h= \left\{v_h\in H^1(\Omega):\ v_h|_T \mbox{ is a linear function}\ \forall\, T \in \mathcal{T}_h\right\}.
\end{equation*}
The spaces $X_h$ and $V_h$ will be employed to approximate the state $u$ and the diffusion coefficient
$q$, respectively. First, we introduce useful operators on $X_h$ and $V_h$. We define the $L^2(\Omega)$ projection $P_h:L^2(\Omega)\to X_h$ by
\begin{equation*}
     (P_h \varphi,\chi) =(\varphi,\chi) , \quad \forall\, \chi\in X_h.
\end{equation*}
Note that the operator $P_h$ satisfies the following error estimates \cite[p. 32]{Thomee:2006}: for any $s\in[1,2]$
\begin{equation}\label{eqn:proj-L2-error}
  \|P_h\varphi-\varphi\|_{L^2(\Omega)} + h\|\nabla(P_h\varphi-\varphi)\|_{L^2(\Omega)}\leq h^s\|\varphi\|_{H^s(\Omega)},\quad \forall \varphi\in H^s(\Omega)\cap H_0^1(\Omega).
\end{equation}
Let $\mathcal{I}_h$ be the Lagrange interpolation operator associated with the finite element space $V_h$. Then it satisfies
the following error estimates for $s=1,2$ and $1 \le p\le \infty$ {(with $sp>d$)} \cite[Theorem 1.103]{ern-guermond}:
\begin{align}
  \|v-\mathcal{I}_hv\|_{L^p(\Omega)} + h\|v-\mathcal{I}_hv\|_{W^{1,p}(\Omega)} & \leq ch^s \|v\|_{W^{s,p}(\Omega)}, \quad \forall v\in W^{s,p}(\Omega).\label{eqn:int-err-inf}
\end{align}
Further, for any $q$, we define a discrete operator $A_h(q):X_h\to X_h$ by
\begin{equation}\label{eqn:Ah}
  (A_h(q)v_h,\chi)=(q\nabla v_h,\nabla \chi),\quad \forall v_h,\chi\in X_h.
\end{equation}

\section{Elliptic case}\label{sec:elliptic}
In this section, we derive error estimates for the elliptic inverse problem.

\subsection{Finite element approximation}

First we describe the regularized formulation and its finite element approximation. To recover
the diffusion coefficient $q$ in the elliptic system \eqref{eqn:pde}, we employ the standard
output least-squares formulation with an $H^1(\Omega)$ seminorm penalty:
\begin{equation}\label{eqn:ob-elliptic}
    \min_{q \in \Uad} {J_\gamma(q)}=\frac12 \|u( q) - z^\delta\|_{L^2(\Omega)}^2  + \frac\gamma2\|\nabla q \|_{L^2(\Omega) }^2,
\end{equation}
where {the admissible set $\Uad$ is defined by \eqref{eqn:admset}} and $u(q)\in H_0^1(\Omega)$ satisfies the variational problem
\begin{equation}\label{eqn:var-elliptic}
 (q\nabla u(q),\nabla v) = (f,v),\quad \forall  v\in H_0^1(\Omega).
\end{equation}

The $H^1(\Omega)$ seminorm penalty is suitable for recovering a smooth diffusion coefficient. The scalar
$\gamma>0$ is the regularization parameter, controlling the strength of the penalty \cite{ItoJin:2015}.
Using standard argument in calculus of variation, it can be verified that for every $\gamma>0$, problem \eqref{eqn:ob-elliptic}--\eqref{eqn:var-elliptic} has at least one
global minimizer $q^*$, and further the sequence of minimizers converges subsequentially in $H^1(\Omega)$
to a minimum seminorm solution as the noise level $\delta$ tends to zero, provided that $\gamma$ is
chosen appropriately in accordance with $\delta$, i.e., $\lim_{\delta\to0^+}\gamma(\delta)^{-1}\delta^2
=\lim_{\delta\to0^+}\gamma(\delta)=0$; see, e.g., \cite{EnglKunischNeubauer:1989,ItoJin:2015}.
{In this work, we focus on the \textit{a priori} choice $\gamma\sim \delta^2$ (cf. Remark
\ref{rmk:rate-ell} below), which is generally sufficient to ensure the noise level condition
$\eqref{eqn:noise-ell}$. In practice, one may also employ \textit{a posteriori} rules. One popular choice
is the discrepancy principle \cite{Morozov:1966,ItoJin:2015}: given some $\tau>1$, it determines the
largest $\gamma>0$ such that
\begin{equation*}
  \|u(q^*) -z^\delta\|_{L^2(\Omega)}\leq \tau \delta,
\end{equation*}
in line with the \textit{a priori} knowledge \eqref{eqn:noise-ell}.}

Now we can formulate the finite element discretization of problem \eqref{eqn:ob-elliptic}--\eqref{eqn:var-elliptic}:
\begin{equation}\label{eqn:ob-disc-elliptic}
    \min_{q_h \in \Uad_h} J_{\gamma,h}(q_h)=\frac12  \|u_h(q_h) - z ^\delta\|_{L^2(\Omega)}^2  + \frac\gamma2\|\nabla q_h \|_{L^2(\Omega) }^2,
\end{equation}
subject to $q_h\in\Uad_h$ and $u_h(q_h)$ satisfying
\begin{align}\label{eqn:fem-elliptic}
(q_h   \nabla u_h(q_h), \nabla v_h) = (f,v_h),\quad \forall v_h \in X_h.
\end{align}
 The discrete admissible set $\Uad_h$ is taken to be
\begin{equation}\label{eqn:admset-h}
    \Uad_h:=\Uad \cap V_h=\{ q_h\in V_h:~~c_0\le q_h(x) \le c_1 \text{ in }\Omega  \}.
\end{equation}
For the discrete problem \eqref{eqn:ob-disc-elliptic}-\eqref{eqn:fem-elliptic},
the following existence and convergence results hold. For any fixed $h>0$, there
exists at least one minimizer $q_h^*$ to problem \eqref{eqn:ob-disc-elliptic}-\eqref{eqn:fem-elliptic}.
Further, the sequence of minimizers $\{q_h^*\}_{h>0}$ contains a subsequence that
converges in $H^1(\Omega)$ to a minimizer to problem \eqref{eqn:ob-elliptic}--\eqref{eqn:var-elliptic}.
The proof follows by a standard argument from calculus of variation and the density
of the space $V_h$ in $H^1(\Omega)$, thus it is omitted; see \cite{Zou:1998,HinzeKaltenbacher:2018} 
for related analysis.

\subsection{Error estimates}

Now we establish an error estimate of the numerical approximation \eqref{eqn:ob-disc-elliptic}-\eqref{eqn:fem-elliptic}
with respect to the exact conductivity $q^\dag$. We shall make the following assumption on the problem data.
\begin{assumption}\label{ass:elliptic}
The exact diffusion coefficient $q^\dag$ and source term $f$ satisfy
{$q^\dag \in   H^2(\Omega) \cap W^{1,\infty}(\Omega) \cap \mathcal A$} and $f \in L^\infty(\Omega)$.
\end{assumption}

Under Assumption \ref{ass:elliptic}, there holds
(see \cite[Lemma 2.1]{LiSun:2017} and \cite[Theorems 3.3 and 3.4]{GruterWidman:1982})
\begin{equation}\label{eqn:reg-u}
   u\in H_0^1(\Omega)\cap H^2(\Omega)\cap W^{1,\infty}(\Omega).
\end{equation}
{Note that this regularity result requires only $W^{1,\infty}(\Omega)\cap \mathcal{A}$.}

The following \textit{a priori} estimate holds. The proof is identical with that for \cite[Lemma 5.2]{WangZou:2010},
but with the Dirichlet boundary condition in place of the Neumann one (see also Lemma \ref{lem:est-01} for
related argument). {The proof requires the estimate $\|q - \mathcal{I}_h q\|_{L^2(\Omega)}\le ch^2$,
due to the use of $u_h(\mathcal{I}_hq)$ as an intermediate solution, and thus the condition $q^\dag\in H^2(\Omega)$
in Assumption \ref{ass:elliptic}. See the proof in Lemma A.1 and \cite[Lemma 5.2]{WangZou:2010} for details.}
\begin{lemma}\label{lem:ellip-est-1}
Let $q^\dag\in \mathcal{A}$ be the exact diffusion coefficient, $u(q^\dag)$ the solution to problem \eqref{eqn:var-elliptic}, and Assumption
\ref{ass:elliptic} be fulfilled. Let $(q_h^*,u_h(q_h^*)) \in \Uad_h
\times X_h$ be a solution of problem \eqref{eqn:ob-disc-elliptic}--\eqref{eqn:fem-elliptic}. Then there holds
$$ \| u_h(q_h^*) - u(q^\dag)  \|_{L^2(\Omega)} + \gamma^\frac12 \| \nabla q_h^* \|_{L^2(\Omega)} \le c(h^2 +\delta+\gamma^{\frac12}).  $$
\end{lemma}

Now we state the main result of this section, i.e., a weighted error estimate for the Galerkin approximation $q_h^*$.
The positivity of the weight $q^\dag | \nabla u(q^\dag)  |^2 + fu(q^\dag)$ will be discussed below.
\begin{theorem}\label{thm:ellip-1}
Let Assumption \ref{ass:elliptic} be fulfilled. Let $q^\dag$ be the exact diffusion coefficient, $ u(q^\dag)$
the solution to problem \eqref{eqn:var-elliptic}, and $q_h^* \in \Uad_h$ a minimizer of
problem \eqref{eqn:ob-disc-elliptic}-\eqref{eqn:fem-elliptic}. Then with $\eta=h^2 +\delta+\gamma^{\frac12}$, there holds
\begin{equation*}
  \int_\Omega (q^\dag-q_h^*)^2 \big(  q^\dag | \nabla u(q^\dag)  |^2 + fu(q^\dag) \big) \,\d x
  \le  c(h \gamma^{-\frac12}\eta + \min(h + h^{-1} \eta,1)) \gamma^{-\frac12}\eta.
\end{equation*}
\end{theorem}
\begin{proof}
For any test function $\fy\in H_0^1(\Omega)$, we have the following splitting  (with $u=u(q^\dag)$)
\begin{align*}
 ((q^\dag-q_h^*)\nabla u,\nabla\fy)& =((q^\dag-q_h^*)\nabla u,\nabla(\fy-P_h\fy))+ (q^\dag\nabla u-q_h^*\nabla u,\nabla P_h\fy).
\end{align*}
Applying integration by parts and the variational formulations of $u$ and $u_h(q_h^*)$ to the first and second terms, respectively leads to
 \begin{align}\label{eqn:sp-00}
 ((q^\dag-q_h^*)\nabla u,\nabla\fy)&=-(\nabla\cdot((q^\dag-q_h^*)\nabla u), \fy-P_h\fy ) + (q_h^*\nabla (u_h(q_h^*) - u),\nabla P_h\fy)\nonumber\\
  &=: {\rm I}_1+{\rm I}_2.
\end{align}
Next we bound the two terms. Direct computation with the triangle inequality gives
\begin{align*}
 \| \nabla\cdot((q^\dag-q_h^*)\nabla u)\|_{L^2(\Omega)} \le & \| \nabla q^\dag\|_{L^\infty(\Omega)}  \| \nabla u \|_{L^2(\Omega)}
 +\| q^\dag-q_h^*\|_{L^\infty(\Omega)}\| \Delta u \|_{L^2(\Omega)}\\
   & +  \| \nabla q_h^*\|_{L^2(\Omega)}  \| \nabla u \|_{L^\infty(\Omega)}.
\end{align*}
In view of the regularity estimate \eqref{eqn:reg-u} and the box constraint of the admissible set $\Uad$, we derive
\begin{equation*}
 \| \nabla\cdot ((q^\dag-q_h^*)\nabla u)\|_{L^2(\Omega)}
 \le c +  \| \nabla q_h^*\|_{L^2(\Omega)}  \| \nabla u \|_{L^\infty(\Omega)}
 \le c (1+  \| \nabla q_h^* \|_{L^2(\Omega)}).
\end{equation*}
This and the Cauchy-Schwarz inequality imply that the term ${\rm I}_1$ is bounded by
\begin{equation*}
 |{\rm I}_1 | \le  c(1+\| \nabla q_h^* \|_{L^2(\Omega)} ) \|  \fy-P_h\fy \|_{L^2(\Omega)}.
\end{equation*}
Now we choose the test function $\fy$ to be $\fy\equiv \frac{q^\dag-q_h^*}{q^\dag} u,$
and then direct computation gives
\begin{equation*}
\nabla \fy = \big(q^{\dag-1} \nabla(q^\dag-q_h^*) - q^{\dag-2}(q^\dag-q_h^*)\nabla q^\dag\big) u + q^{\dag-1}(q^\dag-q_h^*)\nabla u,
\end{equation*}
which implies $\fy\in H_0^1(\Omega)$.
By the box constraint of the admissible set $\Uad$ and the regularity estimate \eqref{eqn:reg-u}, we have
\begin{align*}
  \|\nabla\fy\|_{L^2(\Omega)}&\le c\big[(1+\|\nabla q_h^*\|_{L^2(\Omega)})\|u \|_{L^\infty(\Omega)}
  + \| \nabla u \|_{L^2(\Omega)}\big]\\
     &\le c(1+\|\nabla q_h^*\|_{L^2(\Omega)}).
\end{align*}
Now the approximation property of the projection operator $P_h$ in \eqref{eqn:proj-L2-error} implies
\begin{equation*}
  \|\fy -P_h\fy\|_{L^2(\Omega)} \le ch\| \nabla \fy\|_{L^2(\Omega)} \le ch(1+\|\nabla q_h^*\|_{L^2(\Omega)}).
\end{equation*}
Thus, in view of  Lemma \ref{lem:ellip-est-1}, the term ${\rm I}_1 $ in \eqref{eqn:sp-00} can be bounded by
\begin{align}\label{eqn:I1-elliptic}
  |{\rm I}_1 | & \le  ch (1+\| \nabla q_h ^*\|_{L^2(\Omega)} )^2 \le  c    h(1+\gamma^{-1}\eta^2) \le  c    h \gamma^{-1}\eta^2.
\end{align}
For the term ${\rm I}_2$, by the triangle inequality, inverse inequality on the space $X_h$, the $L^2(\Omega)$ 
stability of $P_h$ and Lemma \ref{lem:ellip-est-1}, we have
\begin{align*}
\|  \nabla(u - u_h(q_h^*)) \|_{L^2(\Omega)} & \leq \|  \nabla(u - P_hu ) \|_{L^2(\Omega)} + h^{-1}\|  P_h u  - u_h(q_h^*)  \|_{L^2(\Omega)}\\
  & \leq c(h + h^{-1}\| u - u_h(q_h^*)   \|_{L^2(\Omega)})\le  c(h + h^{-1}\eta).
\end{align*}
Meanwhile, clearly, there holds $\|  \nabla(u - u_h(q_h^*)) \|_{L^2(\Omega)} \le c$, and hence,
the Cauchy-Schwarz inequality and Lemma \ref{lem:ellip-est-1} imply
\begin{align}
  {\rm I}_2  & \le    \|  \nabla(u - u_h(q_h^*)) \|_{L^2(\Omega)} \|  \nabla \fy\|_{L^2(\Omega)}\nonumber\\
   &\le c \min(h + h^{-1}\eta, 1)(1+\|\nabla q_h^* \|_{L^2(\Omega)})\label{eqn:I2-elliptic}\\
   &\le c \min(h^{-1}\eta, 1) \gamma^{-\frac12}\eta.\nonumber
\end{align}
The  estimates \eqref{eqn:I1-elliptic} and \eqref{eqn:I2-elliptic} together imply
\begin{align*}
 ((q^\dag-q_h^*)\nabla u,\nabla\fy)
  \le  c(h \gamma^{-\frac12}\eta + \min(h + h^{-1} \eta,1)) \gamma^{-\frac12}\eta.
\end{align*}
Now we claim the identity
\begin{equation*}
  ((q^\dag-q_h^*)\nabla u,\nabla\fy)  = \frac12\int_\Omega \Big(\frac{q^\dag-q_h^*}{q^\dag}\Big)^2 \big(  q^\dag | \nabla u  |^2 + fu \big) \,\d x,
\end{equation*}
which together with the preceding estimate leads directly to the desired assertion. 
It remains to show the claim. Actually, integration by parts yield
\begin{equation*}
  ((q^\dag-q_h^*)\nabla u,\nabla\fy) = -\Big(\nabla(\frac{q^\dag-q_h^*}{q^\dag}),q^\dag\fy \nabla u\Big)-\Big(\frac{q^\dag-q_h^*}{q^\dag}\fy,\nabla\cdot(q^\dag\nabla u)\Big).
\end{equation*}
By the governing equation for $u$, $f=-\nabla\cdot(q^\dag\nabla u)$, we deduce
\begin{align*}
  ((q^\dag-q_h^*)\nabla u,\nabla\fy) = \frac{1}{2}((q^\dag-q_h^*)\nabla u,\nabla\fy) -\frac12\Big(\nabla(\frac{q^\dag-q_h^*}{q^\dag}),q^\dag\fy \nabla u\Big)+\frac12\Big(\frac{q^\dag-q_h^*}{q^\dag}\fy, f\Big).
\end{align*}
Then plugging $\fy=\frac{q^\dag-q_h^*}{q^\dag}u$ and collecting the terms give the claim and complete the proof.
\end{proof}

The next result gives an $L^2(\Omega)$ error estimate. The
notation $\mathrm{dist}(x,\partial\Omega)$ denotes the distance of $x\in\Omega$ to the
boundary $\partial\Omega$.
\begin{corollary}\label{cor:error-elliptic}
Let Assumption \ref{ass:elliptic} be fulfilled, and assume that there exists some $\beta\geq0$ such that
\begin{equation} \label{eqn:cond-beta}
(q^\dag|\nabla u(q^\dag)|^2 + fu(q^\dag))(x) \ge c\,\mathrm{dist}(x,\partial\Omega)^\beta \quad \text{a.e. in }\Omega.
\end{equation}
Then the approximation $q_h^*$ satisfies
\begin{equation*}
  \|q^\dag-q_h^*\|_{L^2(\Omega)}  \le c((h \gamma^{-\frac12}\eta + \min( h^{-1} \eta,1)) \gamma^{-\frac12}\eta)^{\frac1{2(1+\beta)}}.
\end{equation*}
In particular, for any $\delta>0$, the choices {$\gamma \sim \delta^2$}
and {$h\sim \sqrt{\delta}$} imply
\begin{equation*}
  \|q^\dag-q_h^*\|_{L^2(\Omega)}  \le c\delta^{\frac1{4(1+\beta)}}.
\end{equation*}
\end{corollary}
\begin{proof}
Let $u=u(q^\dag)$. Then it follows from Theorem \ref{thm:ellip-1} that
\begin{equation*}
  \int_\Omega (q^\dag-q_h^*)^2 \big(  q^\dag | \nabla u  |^2 + fu \big) \,\d x
  \leq c(h \gamma^{-\frac12}\eta + \min(h + h^{-1} \eta,1)) \gamma^{-\frac12}\eta. 
\end{equation*}
Then we decompose the domain $\Omega$ into two disjoint sets $\Omega=\Omega_\rho\cup\Omega_\rho^c$:
\begin{equation*}
 \Omega_\rho = \{ x\in\Omega: ~ \text{dist}(x,\partial\Omega) \ge \rho \} \quad \text{and}\quad  \Omega_\rho^c = \Omega\backslash \Omega_\rho.
\end{equation*}
where the constant $\rho>0$ is to be chosen. On the subdomain $\Omega_\rho$, we have
\begin{align*}
 \int_{\Omega_\rho} (q^\dag-q_h^*)^2 \,\d x &\le  \rho^{-\beta} \int_{\Omega_\rho} (q^\dag-q_h^*)^2 \rho^{\beta} \,\d x \\
 &\le  \rho^{-\beta} \int_{\Omega_\rho} (q^\dag-q_h^*)^2 \text{dist}(x,\partial\Omega)^\beta \,\d x\\
 &\le  c \rho^{-\beta} \int_{\Omega_\rho} (q^\dag-q_h^*)^2  \big(  q^\dag | \nabla u  |^2 + fu \big) \,\d x\\
&\le c \rho^{-\beta}(h \gamma^{-\frac12}\eta + \min(h + h^{-1} \eta,1)) \gamma^{-\frac12}\eta. 
\end{align*}
Meanwhile, by the box constraint of $\Uad$ and $\Uad_h$, we have
\begin{align*}
 \int_{\Omega_\rho^c} (q^\dag-q_h^*)^2 \,\d x \le   c |\Omega_\rho^c|  \le c\rho.
\end{align*}
Then the desired result follows directly by balancing $\rho^{-\beta}(h \gamma^{-\frac12}\eta + \min(h + h^{-1} \eta,1)) \gamma^{-\frac12}\eta$
with $\rho$. 
\end{proof}

\begin{remark}\label{rem:err-2}
The positivity condition \eqref{eqn:cond-beta} has been established in \cite[Lemma 3.7]{Bonito:2017} for $\beta=2$,
provided that $\Omega$ is a Lipschitz domain, $q^\dag\in \mathcal{A}$ and $f\in L^2(\Omega)$ with $f \ge c_f >0$.
Moreover, the condition with $\beta=0$ holds provided that the domain $\Omega$ is $C^{2,\alpha}$,
$q^\dag \in C^{1,\alpha}(\overline{\Omega})$ and $f\in C^{0,\alpha}(\overline{\Omega})$, with $\alpha>0$ and
$f \ge c_f >0$ \cite[Lemma 3.3]{Bonito:2017}. In the latter case, by choosing {$\gamma \sim \delta^2$}
and {$h\sim \sqrt{\delta}$}, we obtain
\begin{equation*}
  \|q^\dag-q_h^*\|_{L^2(\Omega)}  \le c\delta^{\frac14}.
\end{equation*}
Qualitatively, this estimate agrees with the conditional stability estimates
in \cite{Bonito:2017}. {Indeed,
for $u({q_1}),u({q_2}) \in H^2(\Omega)\cap H_0^1(\Omega)$ with $\|u({q_1}) - u(q_2)\|_{L^2(\Omega)}\le
\delta$, \cite[Theorem 3.2]{Bonito:2017} implies
\begin{equation*}
    \| q_1- q_2 \|_{L^2(\Omega)} \le c \|  u(q_1) - u(q_2)\|_{H^1(\Omega)}^\frac12.
\end{equation*}
This, the Gagliardo-Nirenberg interpolation inequality \cite{BrezisMironescu:2018}
\begin{equation*}
  \|u\|_{H^1(\Omega)} \leq \|u\|_{L^2(\Omega)}^\frac12\|u\|_{H^2(\Omega)}^\frac12,
\end{equation*}
and the regularity assumption $u(q_1), u(q_2) \in H^2(\Omega)$ directly give
\begin{align*}
   \| q_1- q_2\|_{L^2(\Omega)}   & \le c (\|  u(q_1) \|_{H^2(\Omega)} +\| u(q_2) \|_{H^2(\Omega)})^{\frac14}
      \|  u(q_1) -u(q_2) \|_{L^2(\Omega)}^\frac14\\
      & \le  c \|  u(q_1) - u(q_2) \|_{L^2(\Omega)}^\frac14 \le c \delta^{\frac14}.
\end{align*}
}There is a growing interest in using conditional stability estimates to derive convergence rates 
for continuous regularized formulations for inverse problems; see, e.g., \cite{EggerHofmann:2018} and references
therein. However, analogous results for discretization errors based on conditional stability seem still missing.
\end{remark}

\begin{remark}\label{rem:cond-ellip}
Several alternative structural conditions have been proposed for deriving convergence rates, and it is
instructive to compare these conditions for the elliptic inverse problem. One condition {\rm(}with a Neumann boundary
condition{\rm)} is given by \cite{Falk:1983}
\begin{equation}\label{eqn:cond-falk}
\nabla u (q^\dag)(x) \cdot \mathbf{\nu}  \ge c> 0,\quad \text{a.e. in}~~\Omega,
\end{equation}
for some constant $c$ and constant vector $\mathbf{\nu}$, or the less restrictive condition  \cite{KohnLowe:1988}
$  \max (|\nabla u |, \Delta u  ) \ge c> 0$, a.e. in $\Omega.$  Either condition
implies the positivity condition \eqref{eqn:cond-beta} with $\beta=0$, provided that $u$ and $f$ have the same
sign {\rm(}e.g., by weak maximum principle for elliptic PDEs{\rm)}. Wang and Zou \cite{WangZou:2010}
derived an error estimate under a weaker assumption $\alpha_0 | \nabla u  |^2  \geq  f$
a.e. in $\Omega.$ However, this condition is not positively homogeneous {\rm(}with respect to problem data
$f${\rm)}.
\end{remark}

\begin{remark}\label{rmk:rate-ell}
Falk \cite{Falk:1983} proposed an numerical schemes for the elliptic inverse problem
with a Neumann boundary condition, based on the output least-square formulation,
by looking for $u_h(q_h^*)\in V_h^r$ {\rm(}continuous piecewise polynomials of degree $r${\rm)}
and $q_h^*\in V_h^{r+1}$. If assumption \eqref{eqn:cond-falk} holds, $u \in W^{r+3,\infty}(\Omega)$ and $q^\dag
\in H^{r+1}(\Omega)\cap \mathcal{A}$, then
$$ \| q^\dag - q_h^* \|_{L^2(\Omega)} \le c (h^r + \delta h^{-2}). $$
If $r=1$ and {$h\sim\delta^{\frac13}$}, it implies an error $O(\delta^{\frac13})$.
This better rate is obtained the stronger regularity assumption
$u(q^\dag)\in W^{4,\infty}(\Omega)$ than that in Theorem \ref{thm:ellip-1}.
\end{remark}

\begin{remark}
Theorem \ref{thm:ellip-1} provides guidance for choosing the algorithmic parameters: given
the noise level $\delta$, we may choose {$\gamma\sim \delta^2$}
and {$h\sim\delta^\frac12$}. The choice {$\gamma\sim\delta^2$} differs from the usual condition
for Tikhonov regularization, i.e., ${ \lim_{\delta\to0^+}\frac{\delta^2}{\gamma(\delta)}=0,}$
but it agrees with that with conditional stability {\rm(}see, e.g., \cite[Theorems 1.1 and 1.2]{EggerHofmann:2018}{\rm)}.
\end{remark}

\section{Parabolic case}\label{sec:parabolic}
Now we analyze the parabolic inverse problem. For a function $v(x,t):\Omega\times(0,T)\to \mathbb{R}$, we shall write $v(t)=v(\cdot,t)$
as a vector valued function on $(0,T)$ below.

\subsection{Finite element approximation}
To recover the diffusion coefficient $q^\dag$ in \eqref{eqn:parabolic},
we employ the standard output least-squares formulation with an $H^1(\Omega)$ seminorm penalty:
\begin{equation}\label{eqn:ob}
    \min_{q \in \Uad}{J_\gamma(q)}=\tfrac12  \|u( q) - z^\delta \|_{L^2(T-\sigma,T;L^2(\Omega))}^2 + \tfrac\gamma2\|\nabla q \|_{L^2(\Omega) }^2,
\end{equation}
where the admissible set $\Uad$ is given by \eqref{eqn:admset}, and $u(q)$ satisfies the variational problem
\begin{equation}\label{eqn:var}
   (\partial_t u(q),v)+(q\nabla u(q),\nabla v) = (f,v),\quad \forall  v\in H_0^1(\Omega),t\in(0,T], \quad \mbox{with }\quad u(0)=u_0.
\end{equation}

Now we describe a discretization of problem \eqref{eqn:ob}--\eqref{eqn:var}, based on the Galerkin FEM in space and the backward Euler
method in time. Specifically, we partition the time interval $[0,T]$ uniformly, with grid points $t_n=n\tau$, $n=0,\ldots,N$, and
a time step size $\tau=T/N$. The fully discrete scheme for problem \eqref{eqn:parabolic} reads: Given $U_h^0=P_hu_0\in X_h$, find $U_h^n\in X_h$ such that
\begin{align}\label{eqn:fully-0}
  (\bar \partial_\tau   U_h^n ,\chi)+(q \nabla U_h^n, \nabla \chi)= (f(t_n),\chi),\quad \forall\chi\in X_h,\quad n=1,2,\ldots,N,
\end{align}
where   $\bar\partial_\tau  \varphi^n = \frac{\varphi^n
- \varphi^{n-1}}{\tau}$ denotes the
backward Euler approximation to $\partial_t \varphi(t_n)$ (with the shorthand $\varphi^n=\varphi(t_n)$).
Using operator $A_h(q)$ in \eqref{eqn:Ah}, we rewrite \eqref{eqn:fully-0} as
\begin{equation*}
  \bar \partial_\tau U_h^n +A_h(q) U_h^n = P_hf(t_n), \quad n=1,2,\ldots,N.
\end{equation*}

Then the finite element discretization of problem \eqref{eqn:ob}--\eqref{eqn:var} reads
\begin{equation}\label{eqn:ob-disc-1}
    \min_{q_h \in \Uad_h} J_{\gamma,h,\tau}(q_h)=\tau \sum_{n=N_\sigma}^N  \|U_h^n(q_h) - z_n^\delta\|_{L^2(\Omega)}^2 + \frac\gamma2\|\nabla q_h \|_{L^2(\Omega) }^2,
\end{equation}
with
\begin{equation}\label{eqn:zn}
   z_n^\delta=\tau^{-1}\int_{t_{n-1}}^{t_n}z^\delta(t) \d t,
\end{equation}
where the discrete admissible set $\Uad_h$ is given by \eqref{eqn:admset-h} and $U_h^n(q_h)\in X_h$ satisfies $U_h^0=P_hu_0$ and
\begin{align}\label{eqn:fully-1}
  \bar \partial_\tau U_h^n(q_h) + A_h(q_h)  U_h^n(q_h) = P_hf(t_n), \quad n=1,2,\ldots,N+1.
\end{align}
Throughout, we assume that $N_\sigma = (T-\sigma)/\tau+1$ is an integer.
Analogous to the elliptic case, the following existence and convergence results hold.
If $u_0\in H_0^1(\Omega)$ and $f\in C([0,T];L^2(\Omega))$, for every $h,\tau>0$, there
exists at least one minimizer $q_{h,\tau}^*\in \mathcal{A}_h$ to problem
\eqref{eqn:ob-disc-1}--\eqref{eqn:fully-1}, and furthermore, the sequence of minimizers $\{q_{h,\tau}^*\}_{h,\tau>0}$
contains a subsequence that converges in $H^1(\Omega)$ to a minimizer of problem
\eqref{eqn:ob}--\eqref{eqn:var}, as $h,\tau\to0^+$; see \cite{Gutman:1990,KeungZou:1998}
for a proof.

\subsection{Error estimates}
Now we derive error estimates of approximations $q_h^*$ under the following regularity
condition on the problem data.

\begin{assumption}\label{ass:data2}
The diffusion coefficient $q^\dag$, initial data $u_0$ and source term $f$ satisfy
$q^\dag\in H^2(\Omega)\cap W^{1,\infty}(\Omega)\cap \Uad$,  
$u_0\in H^2(\Omega)\cap H_0^1(\Omega) \cap W^{1,\infty}(\Omega)$, $f\in L^\infty((0,T)\times\Omega)\cap C^1([0,T];L^2(\Omega))\cap W^{2,1}(0,T;L^2(\Omega)).$
\end{assumption}

Under Assumption \ref{ass:data2}, the parabolic problem \eqref{eqn:parabolic} has a unique solution
\begin{equation}\label{reg-parabolic}
u \in W^{1,p}(0,T;L^q(\Omega)) \cap L^p(0,T; W^{2,q}(\Omega)), \quad \forall p,q\in(1,\infty)
\end{equation}
The result follows directly from maximal $L^p$ regularity of the parabolic equation, see e.g., \cite[Lemma 2.1]{LiSun:2017}.
Then by real interpolation and Sobolev embedding theorem, we deduce
\begin{equation}\label{reg-parabolic-inf}
  u \in L^\infty(0,T; W^{1,\infty}(\Omega)).
\end{equation}
Further, there holds \cite[Lemma 3.2]{Thomee:2006}
\begin{equation}\label{reg-parabolic-2}
\| \partial_t u(t) \|_{L^2(\Omega)} + t\| \partial_{tt}u(t) \|_{L^2(\Omega)} \le c,\quad \forall t\in[0,T].
\end{equation}{
The latter estimate immediately implies a useful uniform bound $\|u(t)\|_{H^2(\Omega)}\leq c$, since
\begin{equation*}
  \|A(q^\dag)u(t)\|_{L^2(\Omega)}\leq \|\partial_tu(t)\|_{L^2(\Omega)} + \|f(t)\|_{L^2(\Omega)}\leq c.
\end{equation*}}

With the choice of $z_n^\delta$ in \eqref{eqn:ob-disc-1}, we have the following estimate.
\begin{lemma}\label{lem:noise-p}
Let Assumption \ref{ass:data2} be fulfilled. Then for $z_n^\delta$ defined in \eqref{eqn:zn}, there holds
\begin{align*}
 \tau \sum_{n=N_\sigma}^N  \int_\Omega |u(t_n) - z_n^\delta|^2 \,\d x \le c (\tau^2 + \delta^2).
\end{align*}
\end{lemma}
\begin{proof}
Let {$u_n = \tau^{-1}\int_{t_{n-1}}^{t_n} u(t)\, \d t$}. Then we have
\begin{equation*}
  u(t_n)-u_n = \tau^{-1}\int_{t_{n-1}}^{t_n}u(t_n)-u(t)\d t = \tau^{-1}\int_{t_{n-1}}^{t_n}\int_t^{t_n}\partial_su(s)\d s\d t,
\end{equation*}
and thus by the regularity estimate \eqref{reg-parabolic-2},
\begin{align*}
  \|u(t_n)-u_n\|_{L^2(\Omega)} &\leq \tau^{-1}\int_{t_{n-1}}^{t_n}\int_t^{t_n} \|\partial_su (s)\|_{L^2(\Omega)}\d s\d t\\
    & \le c \tau \|\partial_tu(t) \|_{C([t_{n-1},t_n];L^2(\Omega))} \le c\tau.
\end{align*}
Meanwhile, by the Cauchy-Schwarz inequality, $\tau|u_n|^2\leq \int_{t_{n-1}}^{t_n}u(t)^2\d t$.
The last two estimates, the definition of the noise level $\delta$ {in \eqref{eqn:noise-p0}} and the following stability estimate
\begin{align*}
 \tau \sum_{n=N_\sigma}^N  \int_\Omega |u_n - z_n^\delta|^2 \,\d x& \le  \int_{T-\sigma}^{T} \int_\Omega |u(t) - z^\delta(t)|^2 \,\d x\d t\leq \delta^2
\end{align*}
imply the desired result immediately.
\end{proof}

The next lemma gives error estimates of fully discrete scheme for the direct problem \eqref{eqn:parabolic}:
find $U_h^n(q^\dag)$ satisfying $U_h^0=P_hu_0$ and
\begin{align}\label{eqn:fully-00}
  \bar \partial_\tau (U_h^n(q^\dag)-U_h^0) + A_h(q^\dag) U_h^n(q^\dag) = P_hf(t_n),\quad \quad n=1,2,\ldots,N.
\end{align}
It plays an important role in the error analysis below. The proof is standard but lengthy, and hence deferred to the appendix.
\begin{lemma}\label{lem:err-parabolic}
Let $q^\dag$ be the exact diffusion coefficient and $u\equiv u(q^\dag)$ the solution to problem \eqref{eqn:var},
and $\{U_h^n(q^\dag)\}$  the solution to problem \eqref{eqn:fully-00}. Then
under Assumption \ref{ass:data2}, 
\begin{equation*}
\begin{split}
\| u(t_n)-U_h^n(q^\dag)\|_{L^2(\Omega)} &\le c(\tau  + h^2).
\end{split}
\end{equation*}
\end{lemma}

The next lemma provides an error estimate of the scheme \eqref{eqn:fully-00} corresponding to
the coefficient $\mathcal{I}_hq^\dag$. It slightly relaxes the regularity assumption in
\cite[Lemma 6.1]{WangZou:2010} from $q^\dag\in W^{2,\infty}(\Omega)$ to $q^\dag\in W^{1,\infty}\II\cap H^2(\Omega)$.
The latter is identical with Assumption \ref{ass:elliptic}.
\begin{lemma}\label{lem:err-1}
Let $q^\dag$ be the exact diffusion coefficient, $u\equiv u(q^\dag)$ the solution to problem \eqref{eqn:var},
and  $\{U_h^n(\mathcal{I}_hq^\dag)\}$ the solutions to the scheme
\eqref{eqn:fully-00} with $\mathcal{I}_hq^\dag$. Then
under Assumption \ref{ass:data2}, 
\begin{equation*}
\begin{split}
\tau\sum_{n=1}^N\| u(t_n)-U_h^n(\mathcal{I}_h q^\dag)\|_{L^2(\Omega)}^2 &\le c(\tau^2  + h^4).
\end{split}
\end{equation*}
\end{lemma}
\begin{proof}
Note that $U_h^n(q^\dag)$ and  $U_h^n(\mathcal{I}_h q^\dag)$  respectively satisfy
\begin{align*}
  A_h(q^\dag)^{-1} \bar \partial_\tau U_h^n(q^\dag) + U_h^n(q^\dag) &= A_h(q^\dag)^{-1} P_h f(t_n),\quad n=1,2\ldots,N,\\
 A_h(\mathcal{I}_h q^\dag)^{-1} \bar \partial_\tau  U_h^n(\mathcal{I}_h  q^\dag) + {U_h^n(\mathcal{I}_h q^\dag)} &= A_h(\mathcal{I}_h  q^\dag)^{-1} P_h f(t_n),\quad n=1,2,\ldots, N,
\end{align*}
with $U_h^0(q^\dag) = U_h^0(\mathcal{I}_hq^\dag) = P_h u_0$. Hence, $\rho_h^n:=U_h^n(q^\dag)  - U_h^n(\mathcal{I}_h q^\dag)$ satisfies
\begin{equation}\label{eqn:err-eq-1}
 A_h(q^\dag)^{-1}  \bar \partial_\tau \rho_h^n + \rho_h^n = (A_h(q^\dag)^{-1}  -
  A_h(\mathcal{I}_h  q^\dag)^{-1}) \big(P_h f(t_n) - \bar \partial_\tau U_h^n(\mathcal{I}_h  q^\dag) \big), \quad n=1,\ldots,N,
\end{equation}
with $\rho_h^0 = 0$. It follows from direct computation that
{\begin{equation*}
\begin{split}
(\bar\partial_\tau A_h(q^\dag)^{-1} \rho_h^n , \rho_h^n)
&= (\bar\partial_\tau A_h(q^\dag)^{-\frac12} \rho_h^n , A_h(q^\dag)^{-\frac12} \rho_h^n) \\
&= \tfrac12\bar
\partial_\tau\|  A_h(q^\dag)^{-\frac12} \rho_h^n \|_{L^2\II}^2 + \tfrac1{2\tau} \|  A_h(q^\dag)^{-\frac12} (\rho_h^n -\rho_h^{n-1}) \|_{L^2\II}^2\\
&\ge \tfrac12\bar
\partial_\tau\|  A_h(q^\dag)^{-\frac12} \rho_h^n \|_{L^2\II}^2,
\end{split}
\end{equation*}}
Then taking inner product \eqref{eqn:err-eq-1} with $\rho_h^n$ and by the Cauchy--Schwarz inequality, we obtain
\begin{equation*}
\begin{split}
 & \tfrac12 \bar\partial_\tau \| A_h(q^\dag)^{-\frac12}  \rho_h^n \|_{L^2(\Omega)}^2 +  \|  \rho_h^n\|_{L^2(\Omega)}^2\\
  \le&  \|   (A_h(q^\dag)^{-1}  - A_h(\mathcal{I}_h  q^\dag)^{-1})  \big(P_h f(t_n) - \bar \partial_\tau U_h^n(\mathcal{I}_h  q^\dag)\big)\|_{L^2(\Omega)} \| \rho_h^n  \|_{L^2(\Omega)}\\
  \le&\tfrac12  \|   (A_h(q^\dag)^{-1}  - A_h(\mathcal{I}_h  q^\dag)^{-1})  \big(P_h f(t_n) - \bar \partial_\tau U_h^n(\mathcal{I}_h  q^\dag)\big)\|_{L^2(\Omega)}^2
  +\tfrac12\| \rho_h^n  \|_{L^2(\Omega)}^2.
\end{split}
\end{equation*}
Further, by Lemma \ref{lem:est-01}, we have for any $\epsilon>0$ and $p\ge \max(d+\epsilon,2)$,
\begin{equation*}
   \| A_h(\mathcal{I}_h q^\dag)^{-1} - A_h(q^\dag)^{-1}  \|_{L^p(\Omega)\rightarrow L^2(\Omega)} \le c h^2.
\end{equation*}
Hence, for $p\ge \max(d+\epsilon,2)$,
{\begin{equation*}
\begin{split}
   \bar\partial_\tau \| A_h(q^\dag)^{-\frac12}  \rho_h^n \|_{L^2(\Omega)}^2 + \|  \rho_h^n\|_{L^2(\Omega)}^2
&\le ch^4 \| {P_h  f(t_n)} - \bar \partial_\tau U_h^n(\mathcal{I}_h  q^\dag) \|_{L^p(\Omega)}^{2} \\
&\le ch^4 \|  f(t_n) - \bar \partial_\tau U_h^n(\mathcal{I}_h  q^\dag) \|_{L^p(\Omega)}^{2},
\end{split}
\end{equation*}
where in the second line we have used the $L^p(\Omega)$ stability of $P_h$ \cite{DouglasDupontWahlbin:1974}.}
Then, summing over $n$ gives
\begin{equation*}
   \| A_h(q^\dag)^{-\frac12}  \rho_h^N \|_{L^2(\Omega)}^2 +  \tau \sum_{n=1}^N \|  \rho_h^n\|_{L^2(\Omega)}^2\le
  ch^4 \Big(\tau \sum_{n=1}^N  \|  f(t_n) \|_{L^p(\Omega)}^2  + \tau \sum_{n=1}^N  \| \bar \partial_\tau U_h^n(\mathcal{I}_h  q^\dag) \|_{L^p(\Omega)} ^2\Big).
\end{equation*}
Then the maximal $\ell^p$ regularity for the backward Euler scheme \cite{Ashyralyev:2002} implies
\begin{equation*}
  \tau \sum_{n=1}^N \|  \rho_h^n\|_{L^p(\Omega)}^2\le
  ch^4 \Big(\tau \sum_{n=1}^N  \|   f(t_n) \|_{L^p(\Omega)}^2  +  \| \nabla u_0 \|_{L^p(\Omega)} ^2\Big).
\end{equation*}
Finally, the desired estimate follows from Lemma \ref{lem:err-parabolic} and the {triangle inequality}.
\end{proof}

The next result gives \textit{a priori} bounds on $q_h^*$ and error estimates on
the corresponding approximations $U_h^n(q_h^*)$. This result will play a crucial role in the proof
of Theorem \ref{thm:error-q} below.
\begin{lemma}\label{lem:err-2}
Let $q^\dag$ be the exact coefficient and $u\equiv u(q^\dag)$ the solution to problem \eqref{eqn:var}.
Let $q_h^*\in \Uad_h$ be the solution to problem \eqref{eqn:ob-disc-1}--\eqref{eqn:fully-1}, and
$\{U_h^n(q_h^*)\}_{n=1}^N$ the fully discrete solution to problem \eqref{eqn:fully-1}. Then under
Assumption \ref{ass:data2}, there holds
\begin{equation*}
  \tau \sum_{n=N_\sigma}^N \| U_h^n(q_h^*) - u(t_n) \|_{L^2(\Omega)}^2 + \gamma \| \nabla q_h^* \|_{L^2(\Omega)}^2
  \le c  (\tau^{2} + h^4  + \delta^2+\gamma).
\end{equation*}
\end{lemma}
\begin{proof}
By the minimizing property of $q_h^*\in \Uad_h$, since $\mathcal{I}_h q^\dag \in \Uad_h$, we deduce
\begin{equation*}
J_{\gamma,h,\tau}(q_h^*)\leq J_{\gamma,h,\tau}(\mathcal{I}_hq^\dag).
\end{equation*}
By the triangle inequality, we derive
\begin{align*}
\tau \sum_{n=N_\sigma}^N \|U_h^n(q_h^*) -  u(t_n)\|_{L^2(\Omega)}^2
 &\le  c \tau \sum_{n=N_\sigma}^N \|U_h^n(q_h^*) - z_n^\delta\|_{L^2(\Omega)}^2
  + c \tau \sum_{n=N_\sigma}^N  \| z_n^\delta-u(t_n)\|_{L^2(\Omega)}^2 .
\end{align*}
These two inequalities together imply
\begin{align*}
   & \tau \sum_{n=N_\sigma}^N \| U_h^n(q_h^*) - u(t_n) \|_{L^2(\Omega)}^2 + \gamma \| \nabla q_h^* \|_{L^2(\Omega)}^2 \\
  \leq & c\tau \sum_{n=N_\sigma}^N \| U_h^n(\mathcal{I}_h q^\dag) - z^\delta_n\|_{L^2(\Omega)}^2 + c\gamma\| \nabla \mathcal{I}_hq^\dag\|_{L^2(\Omega)}^2 + c \tau \sum_{n=N_\sigma}^N  \| z^\delta_n-u(t_n) \|_{L^2(\Omega)}^2\\
  \leq &  c \tau \sum_{n=N_\sigma}^N \| U_h^n(\mathcal{I}_h q^\dag) - u(t_n)\|_{L^2(\Omega)}^2
  +   c\gamma\| \nabla \mathcal{I}_hq^\dag\|_{L^2(\Omega)}^2 + c \tau \sum_{n=N_\sigma}^N  \| z^\delta_n-u(t_n) \|_{L^2(\Omega)}^2\\
  \leq &  c \tau \sum_{n=N_\sigma}^N \| U_h^n(\mathcal{I}_h q^\dag) - u(t_n)\|_{L^2(\Omega)}^2
  +   c\gamma\| \nabla \mathcal{I}_hq^\dag\|_{L^2(\Omega)}^2 + c(\delta^2 + \tau^2),
\end{align*}
where the last line follows from Lemma \ref{lem:noise-p}.
Since $q^\dag\in W^{1,\infty}(\Omega)$ by Assumption
\ref{ass:data2}, $\|\nabla \mathcal{I}_hq^\dag\|_{L^2(\Omega)}\leq c$, cf. \eqref{eqn:int-err-inf}.
Combining the preceding estimates with Lemma \ref{lem:err-1} completes the proof.
\end{proof}

Now we give the main result of this section, i.e., error estimate of the numerical approximation
$q_h^*\in \Uad_h$, with the weight involving $q^\dag|\nabla u(t_n)  |^2 +( f(t_n) -\partial_t u(t_n))u(t_n)$.
whose positivity will be analyzed below in Section \ref{ssec:positive}.
\begin{theorem}\label{thm:error-q}
Let $q^\dag$ be the exact diffusion coefficient, $u\equiv u(q^\dag)$ the solution to problem \eqref{eqn:var},
and $q_h^*\in \Uad_h$ a solution to problem \eqref{eqn:ob-disc-1}--\eqref{eqn:fully-1}. Then
under Assumption \ref{ass:data2}, with $\eta= \tau + h^2  + \delta + \gamma^{\frac12}$, there holds
\begin{align*}
 &\tau^3 \sum_{j=N_\sigma+1}^N \sum_{i=N_\sigma+1}^j \sum_{n=i}^j  \int_\Omega \Big(\frac{q^\dag-q_h^*}{q^\dag}\Big)^2 \Big(q^\dag|\nabla u(t_n)  |^2 +( f(t_n) -\partial_t  u(t_n))u(t_n)\Big)\,\d x\\
 \le& {c(h \gamma^{-\frac12}\eta+  \min(1,h^{-1}\eta))\gamma^{-\frac12}\eta}.
\end{align*}
\end{theorem}
\begin{proof}
For any test function $\fy\in H_0^1(\Omega)$, we have
 \begin{align*}
 ((q^\dag-q_h^*)\nabla u(t_n),\nabla\fy)&=-(\nabla\cdot((q^\dag-q_h^*)\nabla u(t_n)), \fy-P_h\fy ) + (q_h^*\nabla (U_h^n(q_h^*) - u(t_n)),\nabla P_h\fy) \nonumber\\
   & \quad + (q^\dag\nabla u(t_n) - q_h^*\nabla U_h^n(q_h^*),\nabla P_h\fy) =\sum_{i=1}^3{\rm I}_i^n.
\end{align*}
Throughout, the test function $\fy$ is taken to be $\fy\equiv \fy^n=\frac{q^\dag-q_h^*}{q^\dag} u(t_n)$.
Then repeating the argument in Theorem \ref{thm:ellip-1} with the regularity estimates \eqref{reg-parabolic} and \eqref{reg-parabolic-inf} and
the approximation property of $P_h$ in \eqref{eqn:proj-L2-error} yields
\begin{equation}\label{eqn:fyn-bdd}
  \|\nabla\fy^n\|_{L^2(\Omega)} 
  \le c(1+\|\nabla q_h^*\|_{L^2(\Omega)})\quad\text{and}\quad
  \| P_h \fy^n - \fy^n \|_{L^2(\Omega)}\le ch(1+\|\nabla q_h^*\|_{L^2(\Omega)}).
\end{equation}
Next we bound the three terms separately. By
Assumption \ref{ass:data2} (and hence the regularity estimates \eqref{reg-parabolic} and \eqref{reg-parabolic-inf}),
and the box constraint of $q^\dag$ and $q_h^*$, the term ${\rm I}_1^n$ is bounded by
\begin{align*}
  |{\rm I}_1^n| & \le  ch (1+\| \nabla q_h ^*\|_{L^2(\Omega)} )^2
  \le  c  h(1+\gamma^{-1}\eta^2) \le  c  h \gamma^{-1}\eta^2,
\end{align*}
For the term ${\rm I}_2^n$, by the triangle inequality, inverse inequality, $L^2(\Omega)$ stability of the operator
$P_h$ in \eqref{eqn:proj-L2-error}, we deduce
\begin{align*}
\|  \nabla(u(t_n) - U_h^n(q_h^*)) \|_{L^2(\Omega)} & \leq \|  \nabla(u(t_n) - P_hu(t_n) ) \|_{L^2(\Omega)} + h^{-1}\|  P_h u(t_n)  - U_h^n(q_h^*)  \|_{L^2(\Omega)}\\
  & \leq c(h + h^{-1}\|P_hu(t_n) - U_h^n(q_h^*)   \|_{L^2(\Omega)}).
\end{align*}
Meanwhile, by the energy argument in Lemma \ref{lem:err-1}, we deduce
\begin{equation*}
  \tau\sum_{n=1}^N\|\nabla U_h^n(q_h^*)\|_{L^2(\Omega)} ^2\leq c\Big(\sum_{n=1}^N\|f(t_n)\|_{L^2(\Omega)}^2+\|\nabla u_0\|_{L^2(\Omega)}^2\Big)\leq c.
\end{equation*}
This and the regularity estimate \eqref{reg-parabolic-2}, $\tau\sum_n\|\nabla(u(t_n)-U_h^n(q_h^*))\|_{L^2(\Omega)}^2\leq c$.
Consequently, the Cauchy-Schwarz inequality, Lemma \ref{lem:err-2}
and \eqref{eqn:fyn-bdd} imply
\begin{align*}
  \tau \sum_{n=N_\sigma}^N  {\rm I}_2^n & \le  \tau \sum_{n=1}^N \|  \nabla(u(t_n) - U_h^n(q_h^*)) \|_{L^2(\Omega)} \|  \nabla \fy^n \|_{L^2(\Omega)}\\
  &\le c \min\Big(1,h+h^{-1}\Big(\tau \sum_{n=N_\sigma}^N \|  u(t_n) - U_h^n(q_h^*)  \|_{L^2(\Omega)}^2\Big)^{\frac12}\Big) (1+\|\nabla q_h^* \|_{L^2(\Omega)})\\
  &\le {c \min(1,h+h^{-1} \eta)\gamma^{-\frac12} \eta\leq c \min(1,h^{-1} \eta)\gamma^{-\frac12}{\eta}.}
\end{align*}
Next we bound the term ${\rm I}_3^n$. It follows from the variational formulations \eqref{eqn:var} and \eqref{eqn:fully-1} that
\begin{align*}
 {\rm I}_3^n &=   (q^\dag\nabla u(t_n) - q_h^*\nabla U_h^n(q_h^*),\nabla P_h\fy^n ) \\
   &= (\bar \partial_\tau  U_h^n(q_h^*)  - \partial_t u(t_n)  , P_h\fy^n )\\
 &=(\bar \partial_\tau  [ U_h^n(q_h^*) -   u(t_n) ],  P_h\fy^n )
 + ( \bar \partial_\tau  u(t_n)  - \partial_t u(t_n)  , P_h\fy^n )=:  {\rm I}_{3,1}^n + {\rm I}_{3,2}^n.
\end{align*}
It remains to bound the two terms ${\rm I}_{3,1}^n$ and ${\rm I}_{3,2}^n$ separately.
Note that
\begin{equation*}
  \bar\partial_\tau u(t_n) -\partial_t u(t_n)  = {\tau^{-1}\int_{t_{n-1}}^{t_n}\partial_su(s) - \partial_tu(t_n)\d s.}
\end{equation*}
Thus, by the regularity estimate \eqref{reg-parabolic-2}, we have for $n\ge 2$
\begin{align*}
  \| \bar \partial_\tau u(t_n)   - \partial_t u(t_n) \|_{L^2(\Omega)}  \le \frac1\tau \int_{t_{n-1}}^{t_n} \int_s^{t_n}
  \|u''(\xi)\|_{L^2(\Omega)}\,\d \xi \,\d s \le c \tau t_{n-1}^{-1} \le c \tau t_{n}^{-1},
\end{align*}
and for $n=1$,
\begin{equation*}
\| \bar \partial_\tau u(\tau)   - \partial_t u(\tau) \|_{L^2(\Omega)}\le c.
\end{equation*}
Consequently, there holds
\begin{equation*}
|{\rm  I}_{3,2}^n| \le  \| \bar \partial_\tau u(t_n)   - \partial_t u(t_n) \|_{L^2(\Omega)}  \| P_h\fy^n \|_{L^2(\Omega)}
           \le c\tau t_n^{-1}, \quad n=1,2,\ldots, N,
\end{equation*}
and
\begin{equation*}
 \Big|\tau^3 \sum_{j=N_\sigma+1}^N \sum_{i=N_\sigma+1}^j \sum_{n=i}^j {\rm I}_{3,2}^n\Big|
 \le  c \tau \int_{T-\sigma}^T \int_{T-\sigma}^t \int_{s}^t \xi^{-1}\,\d\xi\d s\d t\le c\tau.
\end{equation*}
Meanwhile, since $U_h^0(q_h^*)=U_h^0$ and $u(0)=u_0$, the summation by parts formula yields
\begin{align}
 &\tau \sum_{n=i}^j {\rm I}_{3,1}^n=\tau \sum_{n=i}^j  (\bar \partial_\tau  [U_h^n(q_h^*) - u(t_n)],  P_h\fy^n)\label{eqn:splitting-sum3}\\
 =&  (U_h^j(q_h^*) - u(t_j),P_h\fy^j) -  (U_h^{i-1}(q_h^*) - u(t_{i-1}),P_h\fy^i) - \tau \sum_{n=i}^{j-1}  ( U_h^n(q_h^*) - u(t_n),  \bar\partial_\tau P_h\fy^{n+1}).\nonumber
\end{align}
For the first two terms, by Cauchy-Schwarz inequality and H\"{o}lder's inequality, we have
\begin{align*}
&\Big|\tau^2 \sum_{j=N_\sigma+1}^N \sum_{i=N_\sigma+1}^j  (U_h^j(q_h^*) - u(t_j),P_h\fy^j) -  (U_h^{i-1}(q_h^*) - u(t_{i-1}),P_h\fy^i)  \Big| \\
\le& {c \Big(\tau \sum_{n=N_\sigma+1}^N \| U_h^n(q_h^*) - u(t_n) \|_{L^2(\Omega)}^2\Big)^\frac12} \le c \eta,
\end{align*}
since by \eqref{reg-parabolic-2}, $\|P_h\fy^i\|_{L^2(\Omega)}\leq c$.
Meanwhile, by using the regularity estimate \eqref{reg-parabolic-2} and the box constraint, we have
\begin{align*}
   \| \bar\partial_\tau P_h\fy^{n} \|_{L^2(\Omega)} &= \tau^{-1}\|\int_{t_{n-1}}^{t_{n}}P_h\frac{q^\dag-q_h^*}{q_\dag}\partial_tu(t)\d t \|_{L^2(\Omega)}
    \leq c\tau^{-1}\int_{t_{n-1}}^{t_{n}}\| \partial_tu(t)\|_{L^2(\Omega)}\d t \leq c,
\end{align*}
and hence
\begin{align*}
&\Big|\tau^3  \sum_{j=N_\sigma+1}^N \sum_{i=N_\sigma+1}^j   \sum_{n=i}^{j-1}  ( U_h^n(q_h^*) - u(t_n),  \bar\partial_\tau P_h\fy^{n+1})  \Big| \\
\le& c \tau^3  \sum_{j=N_\sigma+1}^N \sum_{i=N_\sigma+1}^j   \sum_{n=i}^{j-1}  \| U_h^n(q_h^*) - u(t_n)\|_{L^2(\Omega)} \\
\le& c \Big(\tau  \sum_{n=N_\sigma+1}^N   \| U_h^n(q_h^*) - u(t_n)\|_{L^2(\Omega)}^2 \Big)^\frac12 \le c\eta.
\end{align*}
Finally, this and the identity
\begin{equation*}
((q^\dag-q_h^*)\nabla u(t_n),\nabla\fy^n) = \frac12\int_\Omega \Big(\frac{q^\dag-q_h^*}{q^\dag}\Big)^2\big(q^\dag|\nabla u(t_n)  |^2 +(f(t_n)-\partial_t u(t_n))u(t_n)\big)\,\d x
\end{equation*}
(cf. the proof of Theorem \ref{thm:ellip-1})
lead to the desired assertion, completing the proof.
\end{proof}

The next result gives an $L^2(\Omega)$ estimate under a suitable positivity condition
similar to \eqref{eqn:cond-beta}. The proof is identical with that for Corollary
\ref{cor:error-elliptic}, and thus omitted.
\begin{corollary}\label{cor:error-q}
{Let Assumption \ref{ass:data2} be fulfilled,} and there exists some $\beta\geq0$ such that
\begin{equation} \label{eqn:cond-beta-2}
q^\dag(x) |\nabla u(q^\dag)(x,t)|^2 + (f(x,t)-\partial_tu(q^\dag)(x,t))u(q^\dag)(x,t) \ge c\,\mathrm{dist}(x,\partial\Omega)^\beta \quad \text{a.e. in }\Omega,
\end{equation}
for any $t\in[T-\sigma,T]$.
Then for any $\delta>0$, with $\eta= \tau + h^2  + \delta + \gamma^{\frac12}$, there holds
\begin{equation*}
  \|q^\dag-q_h^*\|_{L^2(\Omega)}\leq c ((h \gamma^{-1}\eta  +  { \gamma^{-\frac12}\min(1,h^{-1}\eta)})\eta)^\frac{1}{2(1+\beta)}.
\end{equation*}
In particular, the choices {$\gamma \sim \delta^2$}, {$h\sim\sqrt{\delta}$} and {$\tau\sim\delta$} imply
\begin{equation*}
  \|q^\dag-q_h^*\|_{L^2(\Omega)}  \le c\delta^{\frac1{4(1+\beta)}}.
\end{equation*}
\end{corollary}

\begin{remark}
{Note that in the identity \eqref{eqn:splitting-sum3}, the first two terms cannot be bounded directly, since
only $\ell^2$ bounds are available on $(U_h^j(q_h^*) - u(t_j),P_h\fy^j)$ and $(U_h^{i-1}(q_h^*) - u(t_{i-1}),P_h\fy^i)$.
The triple sum $\sum_{j=N_{\sigma}+1}^N \sum_{i=N_{\sigma}+1}^j \sum_{n=i}^j$ is precisely to exploit relevant $\ell^2$ bounds.}
\end{remark}

\begin{remark}
The error estimate in Corollary \ref{cor:error-q} provides the usual $L^2(\Omega)$ error estimate.
Alternatively, one obtains the $L^2(\Omega)$ estimate, if the following
structural condition holds: For the exact diffusion coefficient $q^\dag$ and the corresponding state variable $u\equiv u(q^\dag)$, there holds
\begin{equation}\label{eqn:ass-uq}
 \int_{T-\sigma}^T \int_{T-\sigma}^t \int_s^t \Big(q^\dag|\nabla u(\xi)  |^2 +( f(\xi) -\partial_\xi u(\xi))u(\xi)\Big)\d \xi \d s\d t > c_0 \qquad \text{a.e. }x \in\overline{\Omega}.
\end{equation}
Similar structural conditions have been assumed in the literature, e.g., the following
characteristic condition \cite{TaiKarkkarinen:1995}:
$$t^{-1}\int_0^t\nabla u(q^\dag)(x,s)\d s\cdot \nu\geq c>0\quad \text{a.e. in}~~
 \Omega\times (0,T),$$
with some constant $c$ and vector $\mathbf{\nu}$, or \cite[Theorem 6.4]{WangZou:2010}
$$\alpha_0 |\int_0^t \nabla u(q^\dag)(s)\d s|^2 + t\int_0^t (\partial_su(q^\dag)(s)-f(s))\d s \geq 0\quad \text{a.e. in}~~
 \Omega\times (0,T).$$
\end{remark}

\subsection{On the positivity condition \eqref{eqn:cond-beta-2}}\label{ssec:positive}
Condition \eqref{eqn:cond-beta-2} allows deriving an $L^2(\Omega)$ estimate,
cf. Corollary \ref{cor:error-q}. Now we give sufficient conditions on problem data
to ensure \eqref{eqn:cond-beta-2}.
\begin{proposition}\label{lem:beta2-parab}
Let $\Omega$ be a bounded Lipschitz domain,
$q^\dag \in \mathcal{A}\cap W^{1,\infty}(\Omega)$, $u_0 \in H^2(\Omega)\cap H_0^1(\Omega)$, and $f\in H^{1}((0,T);L^2(\Omega))$.
Meanwhile, assume that $f \ge c_f >0$ and $\partial_tf \le 0$ a.e. in $\Omega\times(0,T)$,
and $u_0(x)\ge0$, $f(x,0) + \nabla\cdot(q^\dag \nabla u_0(x)) \le 0$ a.e. in $\Omega$. 
Then the positivity condition \eqref{eqn:cond-beta-2} holds with $\beta = 2$, with the constant $c$ only depending
on $c_0, c_1, c_f$ and $\Omega$.
\end{proposition}
\begin{proof}
Since $u_0\ge0$ and $f>c_f$, the maximum principle of parabolic equations
\cite{Friedman:1958} implies
\begin{equation*}
 u(x,t) \ge 0,\quad \forall (x,t)\in\overline{\Omega}\times[0,T].
\end{equation*}
Let $w(x,t)=\partial_t u(x,t)$. Then it satisfies
\begin{equation*}
\left\{\begin{aligned}
   \partial_t w - \nabla\cdot(q^\dag \nabla w) &= \partial_tf,\quad \mbox{in }\Omega\times(0,T], \\
   w & = 0, \quad \mbox{on }\partial\Omega\times(0,T],\\
   w(0) &= f(0) + \nabla\cdot(q^\dag \nabla u_0),\quad\mbox{in }\Omega.
\end{aligned}\right.
\end{equation*}
By assumption, $\partial_tf\leq0$ in $\Omega\times(0,T]$ and $w(0)\leq0$ in $\Omega$.
Then the parabolic maximum principle implies $\partial_tu(x,t) = w(x,t) \le 0$ in  $\overline \Omega\times[0,T]$.
Therefore, there holds
 \begin{align}\label{eqn:est-pos}
   q^\dag(x) |\nabla u(x,t)|^2 + (f(x,t)-\partial_tu(x,t))u(x,t)
&\ge \min(c_0, c_f)(|\nabla u(x,t)|^2 + u(x,t)).
\end{align}
So it suffices to prove $u(x,t)\ge c\, \text{dist}(x,\partial\Omega)^2$ for $(x,t)\in\Omega\times(0,T]$. For
any fixed $t\in[T-\sigma,T]$, we have $f(x,t) - \partial_t u(x,t)\in L^2(\Omega)$. Now consider the elliptic problem
\begin{align}\label{eqn:para-ellip}
 \left\{\begin{aligned}
 {-}  \nabla \cdot(q^\dag\nabla u(t)) &= f(t) - \partial_t u(t),\quad\text{in } \Omega,\\
   u(t) &= 0, \quad \text{on }  \partial \Omega.
\end{aligned}\right.
\end{align}
Let $G(x,y)$ be the Green's function corresponding to the elliptic operator $\nabla
\cdot(q^\dag(x)\nabla \cdot)$. Then $G(x,y)$ is nonnegative (by maximum principle)
and satisfies the following \textit{a priori} estimate (see e.g., \cite[Theorem 1.1]{GruterWidman:1982} and \cite[Lemma 3.7]{Bonito:2017})
 \begin{align*}
 G(x,y) \ge c |x-y|^{2-d}\quad \text{for}\quad |x-y|\le \rho(x):=\text{dist}(x,\partial\Omega).
\end{align*}
Consequently, for any $x\in \Omega$ and $t\in [T-\sigma,T]$, there holds
\begin{align*}
 u(x,t) & = \int_\Omega G(x,y)  (f(y,t) - \partial_t u(y,t)) \,\d y \ge \int_\Omega G(x,y)  f(y,t)   \,\d y \ge c_f \int_\Omega G(x,y)   \,\d y  \\
 &\ge  c_f \int_{|x-y|<\rho(x)/2} G(x,y)   \,\d y \ge c \int_{|x-y|<\rho(x)/2} |x-y|^{2-d} \, \d y \ge c\rho(x)^2.
 \end{align*}
This completes the proof of the proposition.
\end{proof}

The next result gives sufficient conditions for the positivity condition \eqref{eqn:cond-beta-2} with
$\beta=0$, under stronger regularity assumptions on the problem data.
\begin{proposition}\label{lem:beta0-parab}
Let $\Omega$ be a bounded $C^{2,\alpha}$domain,
$f\in C^1([0,T]; C^{0,\alpha}(\overline{\Omega}))$
with $f \ge c_f >0$, $\partial_tf \le 0$ in $\overline{\Omega}\times[0,T]$,
and $u_0 \in C^{2,\alpha}(\overline{\Omega})\cap H_0^1(\Omega)$ with $u_0\ge 0$ in $\Omega$.
Moreover, assume $q^\dag \in \mathcal{A}\cap C^{1,\alpha}(\overline{\Omega})$,
and $f(x,0) + \nabla\cdot(q^\dag\nabla u_0(x)) \le 0$ in $\Omega$.
Then the positivity condition \eqref{eqn:cond-beta-2} holds with $\beta = 0$, with the constant only depending
on $c_0, c_1, c_f, \Omega$ and $\| q^\dag \|_{C^{1,\alpha}(\overline{\Omega})}$.
\end{proposition}
\begin{proof}
By the argument in the proof of Proposition \ref{lem:beta2-parab}, we have $\partial_t u \in C([0,T]; C^{0,\alpha}(\overline{\Omega}))$
and $\partial_t u \le 0$ for all $(x,t)\in \overline{\Omega}\times (0,T)$. Hence, the inequality \eqref{eqn:est-pos} is still valid.
Now it suffices to show that for any $(x,t)\in\overline{\Omega}\times[T-\sigma,T]$,
there holds $|\nabla u(x,t)|^2 + u(x,t) \ge c > 0$. Note that $u(x,t)$ is the solution of the elliptic problem \eqref{eqn:para-ellip}
with a $C^{0,\alpha}(\overline{\Omega})$ source term $f(t)-\partial_t u(t)\ge f(t) \ge c_f$.
Then the desired result follows from Schauder estimates and a standard compactness argument.
For the details, see the proof of \cite[Lemma 3.3]{Bonito:2017}.
\end{proof}

\section{Numerical results}\label{sec:numer}

In this section, we present several numerical experiments to complement the analysis.
Throughout, the discrete optimization problem is solved by the conjugate gradient
(CG) method, which converges within tens of iterations. The lower and upper bounds
in the admissible set $\Uad$ are taken to be $0.5$ and $ 5$, respectively, and are enforced
by a projection step. In the elliptic case, the noisy data $z^\delta$ is generated by
\begin{equation*}
z^\delta(x) = u(q^\dag)(x) +  \varepsilon \sup_{x\in\Omega}|u(q^\dag)|\xi(x),
\end{equation*}
where $\xi$ follows the standard Gaussian distribution, and $\varepsilon>0$ denotes the (relative) noise
level, and similarly for the parabolic case. The noisy data $z^\delta$ is first generated on a fine mesh
and then interpolated to a coarse spatial/ temporal mesh for the inversion step. All the computations
are carried out on a personal laptop with MATLAB 2019.

\subsection{Numerical results for elliptic problems}

First we give one- and two-dimensional elliptic examples.

\begin{example}\label{exam:ell1d}
$\Omega=(0,1)$, $q^\dag(x)=2+\sin2\pi x$ and $f\equiv1$. The exact
data is generated on a fine mesh with a mesh size $h=1/3200$.
\end{example}

The numerical results for Example \ref{exam:ell1d} are summarized in Table \ref{tab:exam1},
where the numbers in the last column denote convergence rates with respect to the noise level
$\delta$, i.e., the exponent $\alpha$ in $O(\delta^\alpha)$. In the tables, $e_q$ and $e_u$ are defined by
\begin{align*}
  e_q &= \|q_h^*-q^\dag\|_{L^2(\Omega)}\quad\mbox{and}\quad
  e_u = \|u_h(q_h^*)-u(q^\dag)\|_{L^2(\Omega)},
\end{align*}
respectively. For the convergence with respect to $\varepsilon$, the regularization parameter $\gamma$ and mesh
size $h$ are taken to be $\gamma=c_\gamma \varepsilon^2$ and $h=c_h\varepsilon^\frac12$, respectively, as
suggested by Corollary \ref{cor:error-elliptic}, where the constant $c_\gamma$ is determined by a trial
and error way. Table \ref{tab:exam1}(a) indicates that the error $e_q$ decays
to zero as the noise level $\varepsilon$ decreases to zero, with an empirical rate $O(\delta^{0.76})$.
Meanwhile, the numerical experiment shows that the weight $|\nabla u(q^\dag)|^2+fu(q^\dag) $ in the error
estimate is indeed strictly positive over the domain $\overline\Omega$, even though both components
have vanishing points. Thus, by Theorem \ref{thm:ellip-1} and Corollary \ref{cor:error-elliptic},
the predicted rate is $O(\delta^\frac14)$, which is much lower than the empirical rate $O(\delta^{0.76})$, indicating
the suboptimality of the predicted rate. The error $e_u$ converges slightly faster than first order. See also Fig.
\ref{fig:ell1d} for an illustration of the reconstructions at three different noise levels.

\begin{table}[hbt!]
\centering
\caption{Numerical results for Example \ref{exam:ell1d}: convergence with respect to $\varepsilon$, with $\gamma$
and $h$ initialized to 5.00e-8 and 2.50e-2.\label{tab:exam1}}
\begin{tabular}{c|cccccccc}
  \toprule
   $\varepsilon$ &  5.00e-2  & 3.00e-2 &  1.00e-2  & 5.00e-3 &  3.00e-3 &  1.00e-3 &  5.00e-4 &       \\
   \midrule
      $e_q$      &  2.52e-1  & 2.56e-1 &  8.08e-2  & 4.84e-2 &  4.06e-2 &  1.63e-2 &  8.43e-3 &  0.76 \\
      $e_u$      &  2.10e-3  & 9.89e-4 &  2.54e-4  & 1.20e-4 &  7.45e-5 &  2.06e-5 &  8.46e-6 &  1.16 \\
   \bottomrule
\end{tabular}
\end{table}

\begin{figure}[hbt!]
  \centering
  \begin{tabular}{ccc}
  \includegraphics[width=.33\textwidth]{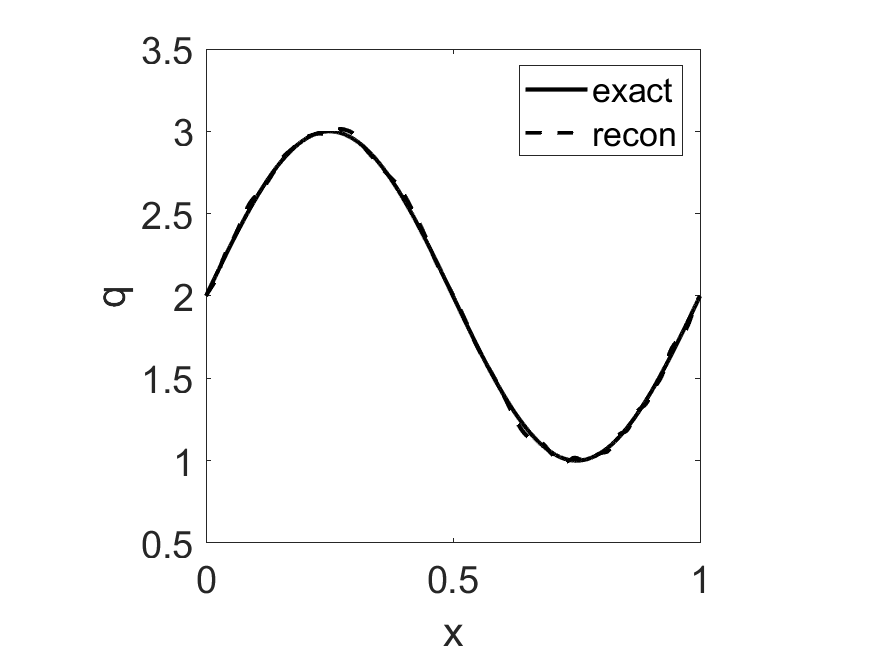}&
  \includegraphics[width=.33\textwidth]{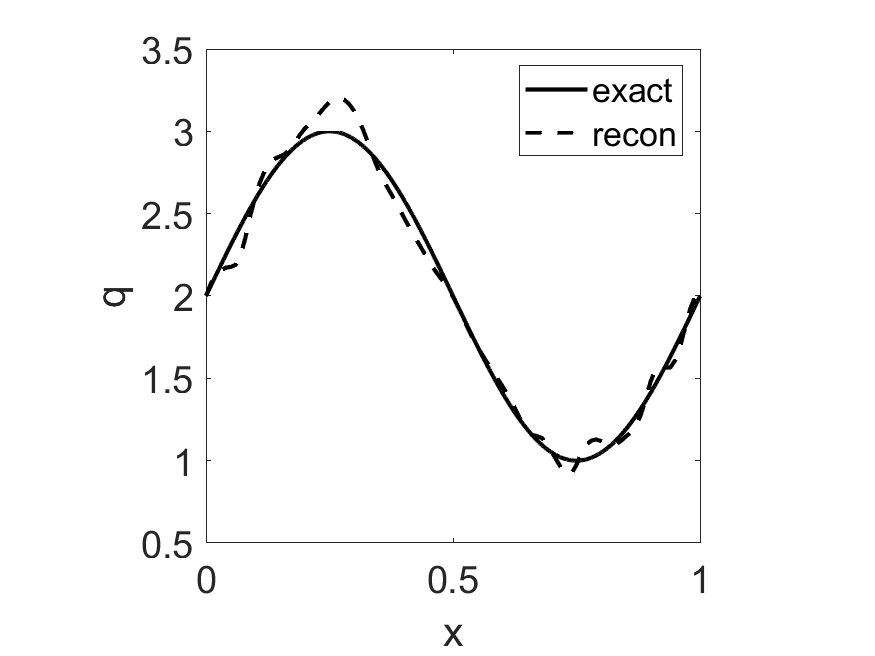}&
  \includegraphics[width=.33\textwidth]{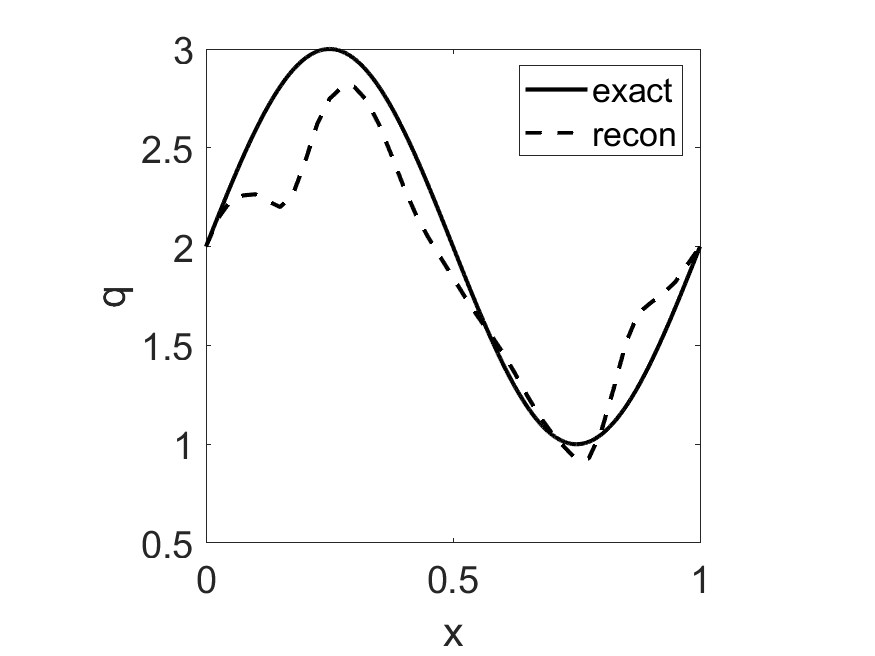}\\
    (a) $\varepsilon$=1e-3 & (b) $\varepsilon$=1e-2  & (c) $\varepsilon$=5e-2
  \end{tabular}
  \caption{Numerical reconstructions for Example \ref{exam:ell1d} at three noise levels.\label{fig:ell1d}}
\end{figure}

\begin{example}\label{exam:ell2d}
$\Omega=(0,1)^2$, $q^\dag(x_1,x_2)=1+x_2(1-x_2)\sin\pi x_1$ and $f\equiv1$. The
data is generated on a fine mesh with a mesh size $h=1/200$.
\end{example}

The numerical results for Example \ref{exam:ell2d} are presented in Table \ref{tab:exam2} and Fig.
\ref{fig:ell2d}. The empirical convergence rates for $e_q$ and $e_u$ with respect to $\delta$ are
about $O(\delta^{0.72})$ and $O(\delta)$, respectively, which are comparable with that for
Example \ref{exam:ell1d}. In either metric, the convergence is very steady. Note that for this
example, the weight $q^\dag|\nabla u(q^\dag)|^2+fu(q^\dag)$ is not strictly positive over
$\overline{\Omega}$, since it vanishes at two corners of the square domain $\Omega$.
\begin{table}
\centering
\caption{Numerical results for Example \ref{exam:ell2d}: convergence with respect to $\varepsilon$, with $\gamma$ and $h$ initialized to
5e-6 and 8.33e-2. \label{tab:exam2}}
\begin{tabular}{c|cccccccc}
  \toprule
   $\varepsilon$ & 5.00e-2 &  3.00e-2 &  1.00e-2 &  5.00e-3 &  3.00e-3 &  1.00e-03 &  5.00e-4 &     \\
   \midrule
      $e_q$      & 4.46e-2 &  3.17e-2 &  1.27e-2 &  6.98e-3 &  5.59e-3 &  2.64e-03 &  1.63e-3 & 0.72 \\
      $e_u$      & 7.88e-4 &  4.11e-4 &  1.20e-4 &  6.56e-5 &  3.89e-5 &  1.39e-05 &  7.72e-6 & 1.00 \\
   \bottomrule
  \end{tabular}
\end{table}

\begin{figure}
  \centering
  \begin{tabular}{ccc}
  \includegraphics[width=.33\textwidth]{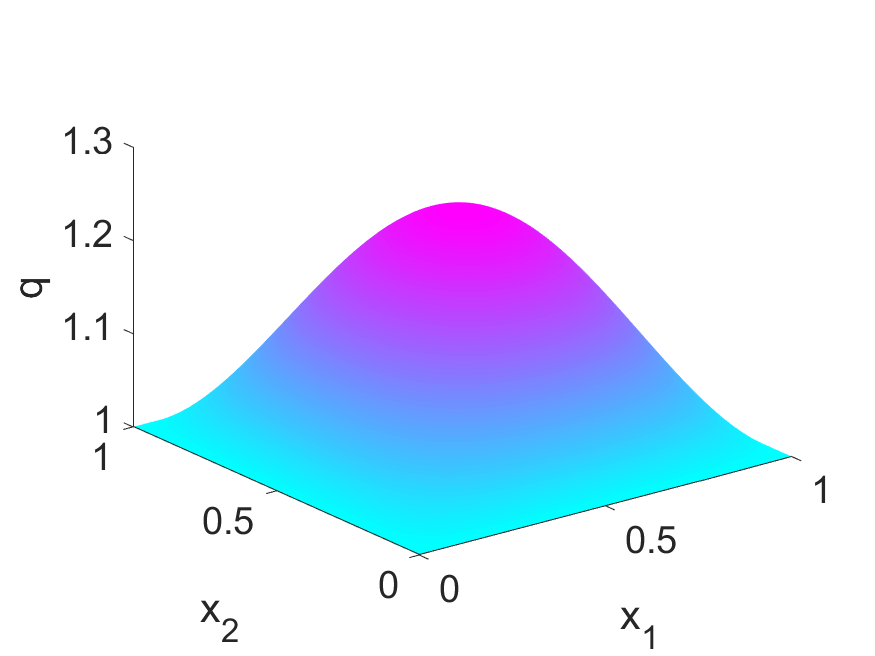}&
  \includegraphics[width=.33\textwidth]{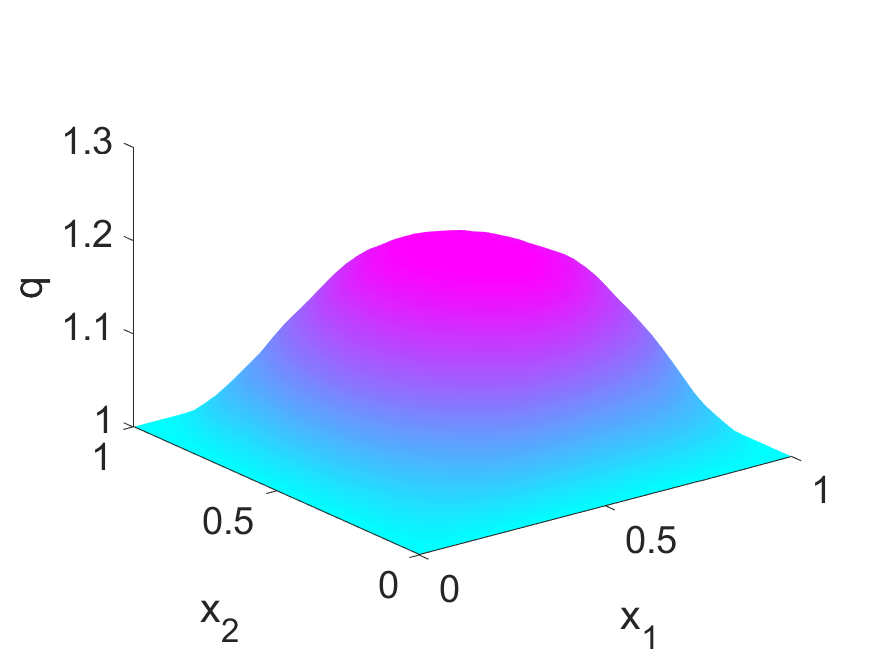}&
  \includegraphics[width=.33\textwidth]{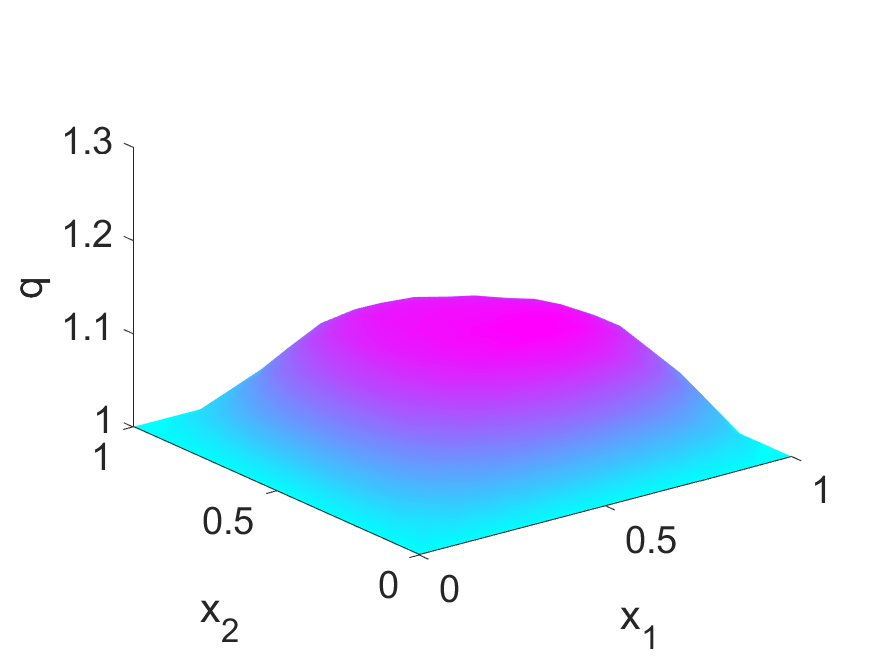}\\
    (a) exact & (b) $\varepsilon$=1e-2  & (c) $\varepsilon$=5e-2
  \end{tabular}
  \caption{Numerical reconstructions for Example \ref{exam:ell2d} at two noise levels.\label{fig:ell2d}}
\end{figure}

\subsection{Numerical results for parabolic problems}
Now we present numerical results for one- and two-dimensional parabolic problems.

\begin{example} \label{exam:par1d}
$\Omega=(0,1)$, $T=0.1$, $\sigma=0$, $q^\dag=2+\sin (2\pi x)e^{-2(1-x)}$, $u_0=\sin (\pi x)$,
and $f=4x(1-x)$. The exact data is generated on a fine mesh with $h=1/1600$ and $\tau=1/8000$.
\end{example}

The numerical results for Example \ref{exam:par1d} are shown in Table \ref{tab:exam3}
and Fig. \ref{fig:par1d}, where $e_q$ is defined as before and $e_u$ is defined by
$e_u=(\tau\sum_{n=N_\sigma}^N\|U_h^n(q_h^*)(t_n)-u(q^\dag)(t_n)\|_{L^2(\Omega)}^2)^\frac12$.
The regularization parameter $\gamma$, the mesh size $h$ and the time step size $\tau$
are chosen such that they all decreases with the noise level $\varepsilon$, as suggested
by Corollary \ref{cor:error-q}. One can check that the positivity condition
\eqref{eqn:cond-beta-2} holds, and thus Corollary \ref{cor:error-q} is indeed applicable.
We observe a very steady convergence for both quantities $e_q$ and $e_u$. The convergence
rate for $e_q$ is comparable with the elliptic cases in Examples
\ref{exam:ell1d} and \ref{exam:ell2d}, however, the rate for $e_u$ is
slightly slower at a rate about $O(\delta^{0.64})$, when compared with
the nearly $O(\delta)$ rate in Examples \ref{exam:ell1d} and \ref{exam:ell2d}.
The precise mechanism for this loss is still unclear.

\begin{table}[hbt!]
\centering
\caption{Numerical results for Example \ref{exam:par1d}, convergence with respect $\varepsilon$, with $\gamma$, $h$ and $\tau$ are initialized with 1.00e-7,
  2.50e-2 and $1/400$. \label{tab:exam3}}
\begin{tabular}{c|cccccccc}
  \toprule
   $\varepsilon$ & 5.00e-2 &  3.00e-2 &  1.00e-2 &  5.00e-3 &  3.00e-3  & 1.00e-3  &  5.00e-4        \\
   \midrule
      $e_q$      & 1.97e-2 &  1.34e-2 &  6.74e-3 &  2.58e-3 &  2.26e-3 &  8.86e-4 &   9.57e-4 & 0.71\\
      $e_u$      & 2.31e-4 &  1.07e-4 &  8.78e-5 &  3.83e-5 &  3.68e-5 &  1.22e-5 &   1.19e-5 & 0.64\\
   \bottomrule
\end{tabular}
\end{table}

\begin{figure}
  \centering
  \begin{tabular}{ccc}
  \includegraphics[width=.33\textwidth]{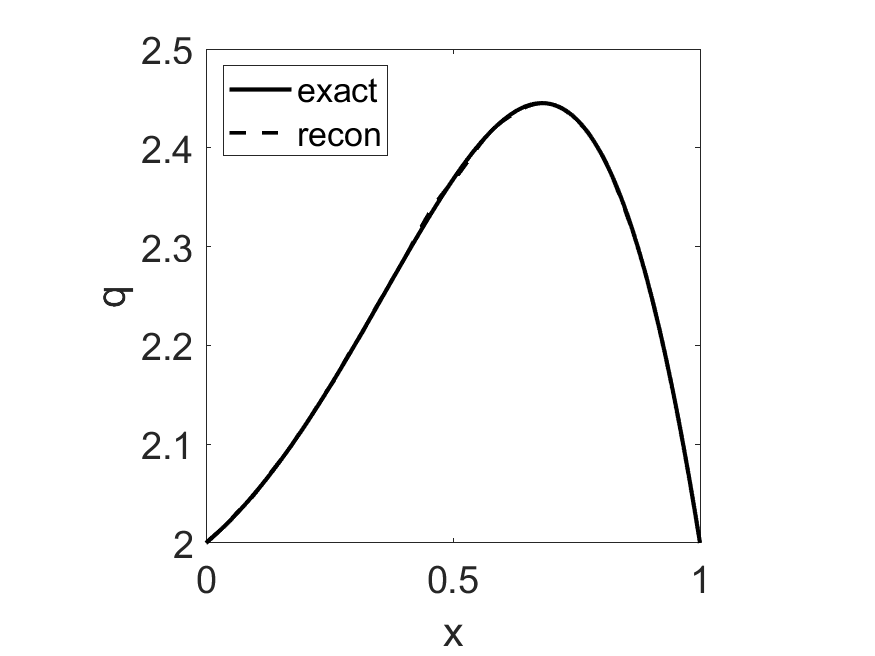}&
  \includegraphics[width=.33\textwidth]{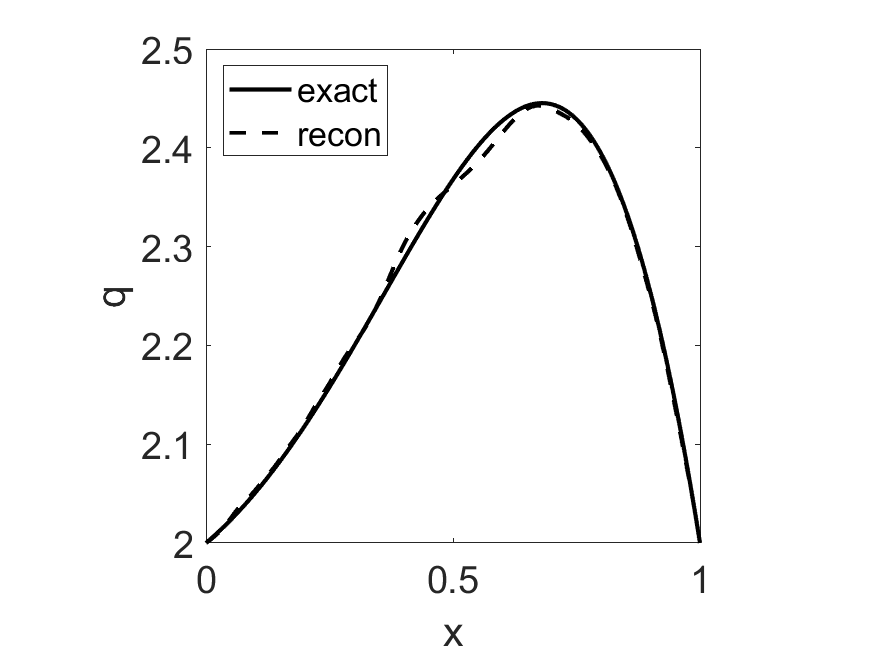}&
  \includegraphics[width=.33\textwidth]{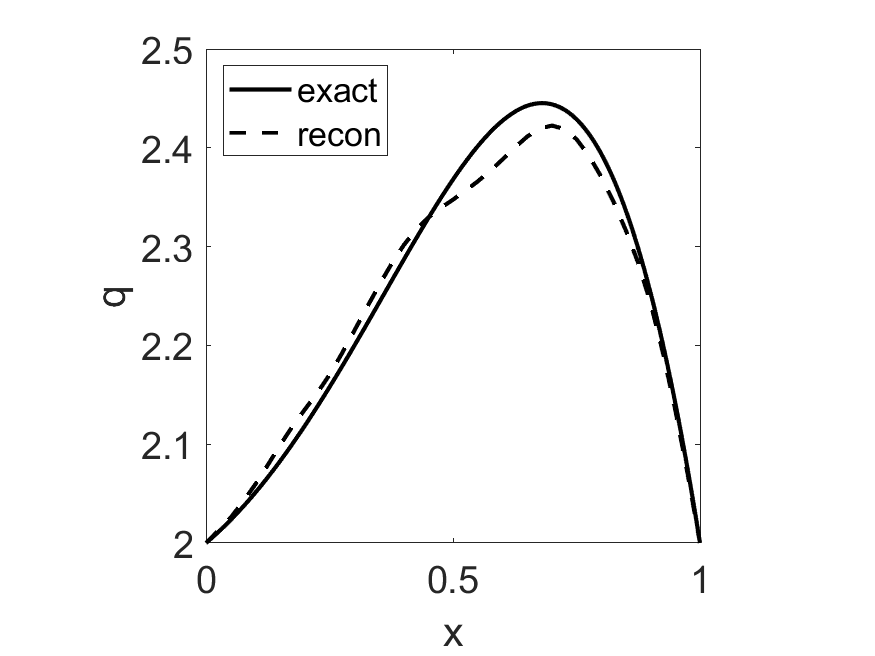}\\
    (a) $\varepsilon$=1e-3 & (b) $\varepsilon$=1e-2  & (c) $\varepsilon$=5e-2
  \end{tabular}
  \caption{Numerical reconstructions for Example \ref{exam:par1d} at three noise levels.\label{fig:par1d}}
\end{figure}

\begin{example}\label{exam:par2d}
$\Omega=(0,1)^2$, $T=0.1$, $q^\dag(x_1,x_2) = 1+(1-x_1)x_1\sin(\pi x_2)$, $u_0(x_1,x_2)=4x_1(1-x_1)$ and $f\equiv1$.
The exact data is generated on a finer mesh with $h=1/200$ and $\tau=1/12800$.
\end{example}

The numerical results for Example \ref{exam:par2d} are shown in Table \ref{tab:exam4} and Fig. \ref{fig:par2d}.
The empirical rates with respect to $\varepsilon$ and $\tau$ are largely comparable with the preceding examples,
and the overall convergence is very steady.

\begin{table}[hbt!]
\centering
\caption{Numerical results for Example \ref{exam:par2d}: convergence with respect to $\varepsilon$, with $\gamma$, $h$ and $\tau$ initialized to
1.00e-6, 8.33e-2 and $1/1600$.\label{tab:exam4}}
 \begin{tabular}{c|cccccccc}
  \toprule
   $\varepsilon$ & 5.00e-2 &  3.00e-2  & 1.00e-2 &  5.00e-3 &  3.00e-3 &  1.00e-3 &  5.00e-4 &      \\
  \midrule
      $e_q$      &   1.95e-2 &  9.54e-3 &  4.32e-3 &  2.87e-3 &  2.28e-3 &  1.37e-3 &  9.37e-4 &  0.62\\
      $e_u$      &   3.49e-3 &  1.70e-3 &  7.52e-4 &  3.92e-4 &  2.68e-4 &  7.26e-5 &  4.17e-5 &  0.94\\
   \bottomrule
  \end{tabular}
\end{table}

\begin{figure}[hbt!]
  \centering
  \begin{tabular}{ccc}
  \includegraphics[width=.33\textwidth]{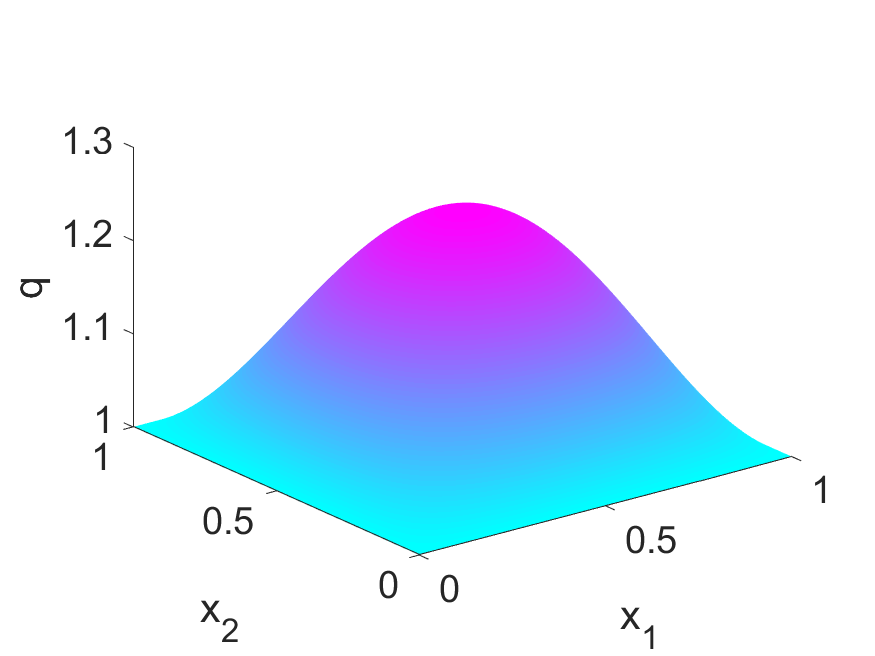}&
  \includegraphics[width=.33\textwidth]{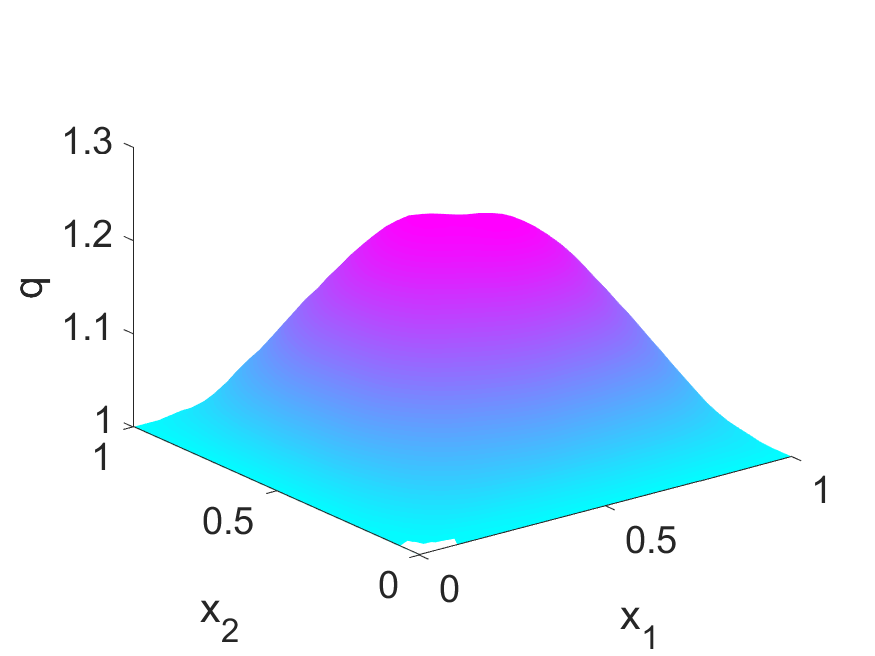}&
  \includegraphics[width=.33\textwidth]{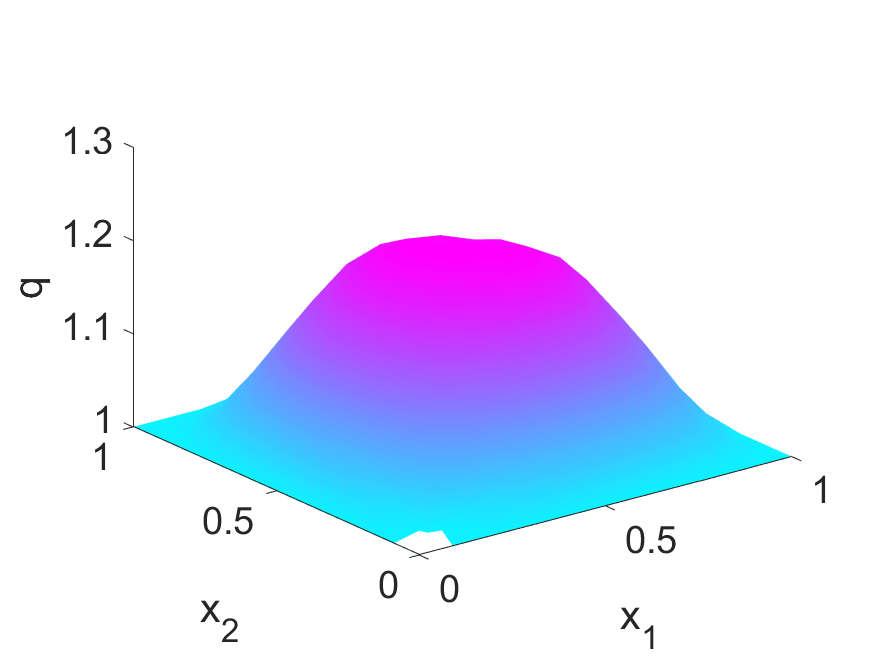}\\
    (a) exact & (b) $\varepsilon$=1e-2  & (c) $\varepsilon$=5e-2
  \end{tabular}
  \caption{Numerical reconstructions for Example \ref{exam:par2d} at two noise levels.\label{fig:par2d}}
\end{figure}

In sum, the numerical experiments confirm the convergence of the Galerkin approximation in
the $L^2(\Omega)$. However, the theoretical rate is still slower than the empirical one. It
remains an important issue to derive sharp error estimates. In addition, it is also of
interest to derive convergence rates with respect to $h$ for the (nonlinear) optimal
control problems (with fixed $\delta$ and $\gamma$), for which there seems no known result.

\appendix
\section{Basic estimates}
We give an error bound on the Galerkin approximation. This estimate is
used in the proof of Lemma \ref{lem:err-1}.
\begin{lemma}\label{lem:est-01}
Let $q\in W^{1,\infty}(\Omega)\cap H^2(\Omega)$, with $c_0\leq q(x)\leq c_1$ a.e. $\Omega$.
Let $u_h(q)\in X_h$ and $u_h(\mathcal{I}_h q)\in X_h$ be the solutions to the variational problems
\begin{align*}
 (q\nabla u_h(q),\nabla v) = (f,v) \quad\text{and}\quad
  (\mathcal{I}_h q \nabla u_h(\mathcal{I}_h q),\nabla v) = (f,v),\quad \forall v\in X_h,
 \end{align*}
respectively. Then for any $\epsilon>0$ and $p\geq \max(d+\epsilon,2)$, there holds
\begin{equation*}
  \| u_h( q) - u_h(\mathcal{I}_h q) \|_{L^2(\Omega)} \le c h^2 \| f \|_{L^p(\Omega)}.
\end{equation*}
\end{lemma}
\begin{proof}
By the definitions of $u_h(\mathcal{I}_h q)$ and $u_h(q)$,
$w_h\equiv u_h(q)-u_h(\mathcal{I}_hq)$ satisfies
\begin{align}\label{eqn:err-01}
  {(q \nabla w_h,\nabla v)} &= ((\mathcal{I}_h q - q)\nabla u_h(\mathcal{I}_h q),\nabla v),\quad\forall v\in X_h.
 \end{align}
Since $q\geq c_0$, by the approximation property \eqref{eqn:int-err-inf}, we derive
\begin{align*}
  c_0 \|\nabla w_h\|_{L^2(\Omega)}^2 & \le (q\nabla w_h,\nabla w_h)  = ((\mathcal{I}_h q-q)\nabla u_h(\mathcal{I}_h q),\nabla w_h)\\
 &\le c_1 \|  \mathcal{I}_h q-q\|_{L^\infty(\Omega)} \| \nabla u_h(\mathcal{I}_h q)\|_{L^2(\Omega)}\| \nabla w_h\|_{L^2(\Omega)}\\
 &\le c h \| q \|_{W^{1,\infty}(\Omega)} \| f \|_{L^2(\Omega)} \| \nabla w_h\|_{L^2(\Omega)},
\end{align*}
i.e., $\|\nabla w_h\|_{L^2(\Omega)}\leq ch\|f\|_{L^2(\Omega)}$.
Next, we derive the $L^2(\Omega)$ estimate by using a duality argument. Let $\psi \in H^2(\Omega)\cap H_0^1(\Omega)$ solve
$(q \nabla v,\nabla \psi) = (v, w_h)$ for any $v\in H_0^1(\Omega)$. Meanwhile, we have
\begin{align*}
  \| w_h \|_{L^2(\Omega)}^2
  &=  (q\nabla w_h,\nabla \psi)
  = (q\nabla w_h,\nabla(\psi-\mathcal{I}_h\psi))
  +(q\nabla w_h,\nabla \mathcal{I}_h\psi)\\
  &\le  \|q\|_{L^\infty(\Omega)}\| \nabla w_h \|_{L^2(\Omega)}\|\nabla (\psi - \mathcal{I}_h\psi)\|_{L^2(\Omega)} + (q\nabla w_h,\nabla  \mathcal{I}_h \psi)\\
   & \le c h^2 \|  \psi \|_{H^2(\Omega)} \| f \|_{L^2(\Omega)} + (q\nabla w_h,\nabla\mathcal{I}_h \psi).
\end{align*}
Further, using \eqref{eqn:err-01} and the \textit{a priori} estimate
$\|u(q)\|_{W^{1,\infty}(\Omega)}\leq c\|f\|_{L^p(\Omega)}$ for any $p\geq\max(d+\epsilon,2)$
\cite[(2.2)]{LiSun:2017}, and the estimate $\|\nabla(u(q)-u_h(q))\|_{L^2(\Omega)}\leq ch$, we obtain
\begin{align*}
 &\quad (q \nabla w_h,\nabla  \mathcal{I}_h \psi)
  =  ((\mathcal{I}_h q- q)\nabla u_h(\mathcal{I}_h q),\nabla  \mathcal{I}_h \psi)\\
 & = ((\mathcal{I}_h q - q)\nabla [u_h(\mathcal{I}_h q) - u(q)],\nabla  \mathcal{I}_h \psi)
 + ((\mathcal{I}_h q - q)\nabla  u(q),\nabla  \mathcal{I}_h \psi)\\
 & \leq \| \mathcal{I}_hq-q\|_{L^\infty(\Omega)}\|\nabla(u_h(\mathcal{I}_hq)-u(q))\|_{L^2(\Omega)}
 \|\nabla\mathcal{I}_h\psi\|_{L^2(\Omega)}\\
 &\quad +  \|\mathcal{I}_hq - q\|_{L^2(\Omega)}\|\nabla u(q)\|_{L^\infty (\Omega)} \|\nabla  \mathcal{I}_h \psi\|_{L^2(\Omega)}\\
 &\le c h^2 \|\nabla  \mathcal{I}_h \psi\|_{L^2(\Omega)} \| f \|_{L^p(\Omega)} \le c  h^2 \|  \psi\|_{H^2(\Omega)} \| f \|_{L^p(\Omega)}.
\end{align*}
This and the regularity $\|\psi\|_{H^2(\Omega)}\leq c\|w_h\|_{L^2(\Omega)}$ lead to
\begin{equation*}
   \| u_h(q) - u_h(\mathcal{I}_h q) \|_{L^2(\Omega)} \le c h^2 \| f \|_{L^p(\Omega)},
\end{equation*}
for any $p \ge \max(d+\epsilon,2)$. This completes the proof of the lemma.
\end{proof}

\section{Proof of Lemma \ref{lem:err-parabolic}}

\begin{proof}
If $f\equiv0$ and $u_0\in H^2(\Omega)\cap H_0^1(\Omega)$, the estimate can be found in \cite[Theorem 3.1]{Thomee:2006}.
It suffices to analyze the case $u_0 = 0$ and $f\in W^{2,1}(0,T;L^2(\Omega))$. Let $A\equiv A(q^\dag):H_0^1\II \to H^{-1}(\Omega)$
by $(Av,\chi)=(q^\dag\nabla v,\nabla \chi)$ for all $\chi\in H_0^1\II$. Then $A$  generates a bounded analytic
semigroup $e^{-At}$ on $L^2\II$ and allows representing the solution $u(t)$ by
\begin{align*}
 u(t)  =  \int_0^t e^{-A(t-s)} f(s) \,\d s.
\end{align*}
Then it follows from integration by parts that
\begin{align*}
   & \|  \partial_t u(t) \|_{L^2\II} + \|  A u(t) \|_{L^2\II}  \le c \| f \|_{C([0,t];L^2\II)} + \int_0^t \|\partial_s f (s)\|_{L^2\II} \,\d s,\\
   &\|  A\partial_t u(t) \|_{L^2\II}   \le c \Big( t^{-1} \| f(0) \|_{L^2\II} + \| f'(t) \|_{L^2\II} + \int_0^t \|\partial_{s}^2f (s)\|_{L^2\II} \,\d s\Big).
\end{align*}
The second inequality and Assumption \ref{ass:data2} imply
\begin{equation*}
  \int_0^t s \|  A \partial_s u(s) \|_{L^2\II} \,\d s  \le c t .
\end{equation*}
Then by the regularity estimate \eqref{reg-parabolic-2} and the approximation property
\eqref{eqn:proj-L2-error}, we derive
\begin{align} \label{eqn:est-ap-01}
 \|  u(t) - P_h u(t) \|_{L^2\II}  \le c h^2 \|  u(t) \|_{H^2\II} \le c h^2 \| Au(t) \|_{L^2\II}.
  \end{align}
Let $u_h$ be the spatially semidiscrete Galerkin approximation, i.e., $\partial_t u_h+ A_h
u_h  = P_hf$ with $u_h(0) = 0$ and $A_h \equiv A_h(q^\dag)$, cf. \eqref{eqn:Ah}.
Then the difference $\zeta(t) = u_h(t) - P_h u(t)$ satisfies
\begin{equation*}
   \partial_t \zeta (t)+ A_h  \zeta(t)= A_h(R_h - P_h) u(t),
\end{equation*}
with $\zeta(0)=0$, where $R_h:H_0^1(\Omega)\rightarrow X_h$ denotes the Ritz projection (associated with $q^\dag$).  Then  \eqref{eqn:proj-L2-error} and
the approximation property of $R_h$ \cite[Lemma 1.1]{Thomee:2006} lead to
\begin{align*}
 t  \zeta(t) &  = t \int_0^t e^{-A_h (t-s)} A_h(R_h-P_h) u(s) \,\d s \\
 & = \int_0^t (t-s)e^{-A_h (t-s)} A_h(R_h-P_h) u(s) \,\d s + \int_0^t  e^{-A_h(t-s)} A_h(R_h-P_h) su(s) \,\d s\\
 & =: {\rm I}_1 + {\rm I}_2.
 \end{align*}
Since $ \| e^{-A_h t} A_h \|_{L^2\II\rightarrow L^2\II}\leq c t^{-1} $, we deduce
\begin{align*}
 \| {\rm I}_1 \|_{L^2\II}\le  \int_0^t  \|(R_h-P_h) u(s)\|_{L^2\II} \,\d s \le  ch^2 \int_0^t  \| Au(s) \|_{L^2(\Omega)} \,\d s \le  c t h^2.
\end{align*}
Similarly, integration by parts allows bounding ${\rm I}_2$ by
\begin{align*}
 \| {\rm I}_2 \|_{L^2\II} & \le  c t\| (R_h -P_h) u(t)\|_{L^2\II} + c\int_0^t \| (R_h - P_h) \partial_s (su(s)) \|_{L^2\II} \,\d s\\
 &\le c t h^2  \| A u \|_{C([0,t];L^2\II)}   + ch^2 \int_0^t s \| A\partial_s(s u(s)) \|_{L^2} \, \d s \le c t h^2.
\end{align*}
The preceding two estimates yield $\| \zeta(t) \|_{L^2(\Omega)} \le c h^2$.
This, \eqref{eqn:est-ap-01}  and the triangle inequality imply
$ \|  u_h(t) - u(t) \|_{L^2\II} \le c h^2 .$
Meanwhile, repeating the argument in \cite[Lemma 4.2]{JinLizhou:nonlinear} yields
$$ \|  u_h(t_n) - U_h^n(q^\dag)  \|_{L^2\II} \le c \tau \Big( \|  f(0) \|_{L^2\II} + \int_0^{t_n} \| \partial_s f(s) \|_{L^2(\Omega)} \,\d s\Big) \le c\tau.  $$
Then the desired assertion follows immediately by the triangle inequality.
\end{proof}

\bibliographystyle{siam}

\end{document}